\documentclass[11pt]{amsart}

\usepackage{fontenc}
\usepackage{mathtools}
\usepackage{amsfonts}
\usepackage{amsopn}
\usepackage{amsmath}
\usepackage{amssymb}
\usepackage{mathrsfs}
\usepackage{etoolbox}
\usepackage{amsthm}
\usepackage{esint}
\usepackage{blindtext}
\usepackage[normalem]{ulem} 
\usepackage{geometry}

\usepackage{hyperref}
\usepackage{cleveref}

\makeatletter
\newcommand*{\rom}[1]{\expandafter\@slowromancap\romannumeral #1@}
\makeatother

\newtheorem{theorem}[subsection]{Theorem}

\newtheorem{lemma}[subsection]{Lemma}

\newtheorem{corollary}[subsection]{Corollary}

\theoremstyle{definition}

\DeclarePairedDelimiter{\abs}{\lvert}{\rvert}
\DeclarePairedDelimiter{\norm}{\lVert}{\rVert}

\makeatletter
\newsavebox\myboxA
\newsavebox\myboxB
\newdimen\mydimen
\newcommand*\xoverline[2][0.75]{%
    \sbox{\myboxA}{$\m@th#2$}%
    \sbox{\myboxB}{\rule{#1\wd\myboxA}{0.4pt}}
    \mydimen=\ht\myboxA%
    \advance\mydimen by 1.5pt%
    \rlap{\hbox to \wd\myboxA{\hss\raisebox{\mydimen}{\usebox\myboxB}\hss}}%
    \usebox\myboxA
}
\makeatother

\newcommand{\iprod}{\mathbin{\lrcorner}}

\numberwithin{equation}{section}

\newtheoremstyle{boldstep}
  {3pt}                
  {3pt}               
  {\normalfont}        
  {}                   
  {\bfseries}          
  {.}                 
  {.5em}              
  {\thmname{#1}\thmnumber{ #2}\thmnote{: #3}}
\theoremstyle{boldstep}
\newtheorem{step}{Step}
\crefname{step}{Step}{Steps}

\title[SKODA DIVISION THEOREM ON COMPACT KÄHLER MANIFOLDS]{SKODA DIVISION THEOREM ON COMPACT KÄHLER MANIFOLDS FOR LINE BUNDLES WITH SINGULAR HERMITIAN METRICS}

\author{Jaehoon Jeong}
\address{Department of Mathematical Sciences, Seoul National University, 1 Gwanak-ro, Gwanak-gu, Seoul 08826, South Korea}
\email{veriche4@snu.ac.kr}
\AddToHook{env/theorem/begin}{\crefalias{subsection}{theorem}}
\AddToHook{env/exercise/begin}{\crefalias{subsection}{exercise}}
\AddToHook{env/definition/begin}{\crefalias{subsection}{definition}}
\AddToHook{env/lemma/begin}{\crefalias{subsection}{lemma}}
\AddToHook{env/conjecture/begin}{\crefalias{subsection}{conjecture}}
\AddToHook{env/corollary/begin}{\crefalias{subsection}{corollary}}
\AddToHook{env/remark/begin}{\crefalias{subsection}{remark}}
\AddToHook{env/example/begin}{\crefalias{subsection}{example}}

\begin{document}

\begin{abstract}
We prove a Skoda division theorem on a compact Kähler manifold for holomorphic sections of line bundles with singular hermitian metrics. As a corollary, we prove a Siu-type division theorem on compact Kähler manifolds for holomorphic line bundles endowed with singular hermitian metrics. 
\end{abstract}
\maketitle
\tableofcontents

\section{Introduction}

Let $\Omega\subset\mathbb{C}^n$ be a pseudoconvex domain and let $\psi$ be a plurisubharmonic function on $\Omega$. Let $g_1,\cdots,g_p$ be holomorphic functions on $\Omega$, and define $\abs{g}^2=\sum_l\abs{g_l}^2$. Let $q=\min(n, p-1)$ and $\alpha>1$. Skoda proved in his paper (\textit{cf.} \cite{Sko72}) that if $f$ is a holomorphic function on $\Omega$ satisfying the integrability condition
\vspace{0.2cm}
\[
\int_\Omega\frac{\abs{f}^2e^{-\psi}}{\abs{g}^{2(\alpha q+1)}}\,dV<\infty{,}
\]
where $dV$ is the standard Lebesgue volume measure on $\mathbb{C}^n$, then there exist holomorphic functions $h_1,\cdots,h_p$ on $\Omega$ such that $f=h_1g_1+\cdots+h_pg_p$ and the $p$-tuple $h=(h_1,\cdots,h_p)$ satisfies the $L^2$ estimate
\vspace{0.2cm}
\[
\int_\Omega\frac{\abs{h}^2e^{-\psi}}{\abs{g}^{2\alpha q}}\,dV\leq\frac{\alpha}{\alpha-1}\int_\Omega\frac{\abs{f}^2e^{-\psi}}{\abs{g}^{2(\alpha q+1)}}\,dV{.}
\] 
In 1978, Skoda generalized this to a geometric setting, which is described by a surjective morphism of vector bundles:
\begin{theorem} \cite{Sko78}
Let $X$ be an $n$-dimensional weakly pseudoconvex Kähler manifold. Let $g: E\to Q$ be a surjective morphism of holomorphic hermitian vector bundles over $X$, where $E$ is Nakano semi-positive. \\
\indent If $M$ is a hermitian line bundle satisfying $\sqrt{-1}\Theta_M\geq q\cdot \sqrt{-1}\Theta_{\det Q}$ for some real number $q> \min\{n, \operatorname{rank}E- \operatorname{rank}Q\}$, then the induced map on global sections
\[
H^0(X, E\otimes K_X\otimes M)\to H^0(X, Q\otimes K_X\otimes M)
\]
is surjective.
\end{theorem}
\indent In algebra, Hilbert's Nullstellensatz relates the common zero locus of a polynomial ideal to the radical of that ideal. In complex analysis, Cartan's Theorem provides an analogue, ensuring that if a function $f$ belongs to the ideal generated by $g_1, \cdots, g_p$ locally at every point, it belongs to the global ideal. However, Cartan's theorem does not tell us what the coefficients $h_l$ in the equation $f = \sum h_l g_l$ actually look like. On the other hand, the Skoda division theorem provides quantitative bounds on the coefficients $h_l$. Skoda proved that if $f$ vanishes to a sufficiently high order at the common zeros of $g$, the division is not only possible, but the solutions are well-behaved. \\
\indent After Skoda's works, there have been many contributions to establish it in various settings. For example, Demailly proved a version for complete Kähler manifolds by using a surjective morphism of holomorphic vector bundles with smooth hermitian metrics (\textit{cf.} \cite[Theorem 11.8]{Dem12}). Chen also proved a Skoda division theorem on complete Kähler domains in $\mathbb{C}^n$ using possibly singular plurisubharmonic weights (\textit{cf.} \cite{Che16}).  Chen used a different type of functional analysis theorem from Skoda's to solve $\xoverline{\partial}\xoverline{\partial}^*_g+\xoverline{\partial}^*_g\xoverline{\partial}$-equations. For smooth projective varieties and essentially Stein manifolds, Varolin proved a related version with holomorphic line bundles possessing singular hermitian metrics (\textit{cf.} \cite[Theorem 1]{Var08}). For further related developments and different perspectives, including precise $L^2$ division theorems and their algebraic applications, we refer the reader to \cite{Ohs02b, Ohs04} and \cite{Siu08, Siu09}. \\
\indent The primary objective of this paper is to establish a Skoda $L^2$ division theorem on compact Kähler manifolds that incorporates singular hermitian metrics, and not just smooth hermitian metrics. Our main result demonstrates that under a natural curvature assumption linking the singular metrics of $F$ and $E$, one can solve the division problem globally with an explicit $L^2$ estimate.

\begin{theorem}\label{MAIN THEOREM 1}
Let $(X,\omega)$ be a compact Kähler manifold of complex dimension $n\geq 1$. Let $F$ and $E$ be holomorphic line bundles with singular hermitian metrics $e^{-\psi}$ and $e^{-\eta}$, respectively. Let $g_1,\cdots,g_p\in H^0(X,E)$. Let $\alpha\in\mathbb{Q}_{>1}$. Let $q=\min\{n,p-1\}$. Assume that 
\vspace{0.2cm}
\[
\sqrt{-1}\partial\xoverline{\partial}\psi\geq\alpha q\cdot \sqrt{-1}\partial\xoverline{\partial}\eta{.}
\]
Then for any $f\in H^0(X,K_X+E+F)$ such that
\vspace{0.2cm}
\[
\int_X\frac{\abs{f}_\omega^2e^{-\psi-\eta}}{(\abs{g}^2e^{-\eta})^{\alpha q+1}}\,dV_\omega<\infty
\]
there exist $p$ sections $h_1,\cdots,h_p\in H^0(X,K_X+F)$ such that
\vspace{0.2cm}
\[
\sum_kh_kg_k=f\,\,\,\text{and}\,\,\,\int_X\frac{\abs{h}_\omega^2e^{-\psi}}{(\abs{g}^2e^{-\eta})^{\alpha q}}\,dV_\omega\leq\Bigl(2+\frac{4(\alpha+1)}{(\alpha-1)^2}\Bigr)\int_X\frac{\abs{f}_\omega^2e^{-\psi-\eta}}{(\abs{g}^2e^{-\eta})^{\alpha q+1}}\,dV_\omega{.}
\]
\end{theorem}

Using \Cref{MAIN THEOREM 1}, we construct another type of Skoda division theorem which can be more convenient to apply in some situations. 

\begin{theorem} \label{MAIN THEOREM 2}
Let $(X,\omega)$ be a compact Kähler manifold of complex dimension $n\geq 1$. Let $F$ and $E$ be holomorphic line bundles with singular hermitian metrics $e^{-\psi}$ and $e^{-\eta}$ with nonnegative curvatures, respectively. That is, $\sqrt{-1}\partial\xoverline{\partial}\psi\geq 0$ and $\sqrt{-1}\partial\xoverline{\partial}\eta\geq 0$. Let $g_1,\cdots,g_p\in H^0(X,E)$ and $q\geq\min\{n,p-1\}$ be an integer. If $\min\{n,p-1\}>0$, then for any $f\in H^0(X,K_X+(q+2)E+F)$ such that
\vspace{0.2cm}
\[
\int_X \frac{\abs{f}_\omega^2e^{-\psi-(q+2)\eta}}{(\abs{g}^2e^{-\eta})^{q+2}}\,dV_\omega <\infty{,}
\]
there exist $p$ sections $h_1,\cdots,h_p\in H^0(X,K_X+(q+1)E+F)$ such that
\vspace{0.2cm}
\[
\sum_kh_kg_k=f\,\,\,\text{and}\,\,\,\int_X \frac{\abs{h}_\omega^2e^{-\psi-(q+1)\eta}}{(\abs{g}^2e^{-\eta})^{q+1}}\,dV_\omega \leq(8q^2+4q+2)\int_X \frac{\abs{f}_\omega^2e^{-\psi -(q+2)\eta}}{(\abs{g}^2e^{-\eta})^{q+2}}\,dV_\omega{.}
\]
If $\min\{n,p-1\}=0$, then for any $f\in H^0(X, K_X+ (q+2)E+ F)$ satisfying
\[
\int_X \frac{\abs{f}^2_\omega e^{-\psi -(q+2)\eta}}{\abs{g_1}^2 e^{-\eta}}\, dV_\omega< \infty{,}
\]
there exists a section $h_1\in H^0(X, K_X+ (q+1)E+ F)$ such that $f= g_1h_1$ and
\[
\int_X \abs{h_1}_\omega^2 e^{-\psi -(q+1)\eta}\, dV_\omega= \int_X \frac{\abs{f}_\omega^2 e^{-\psi -(q+2)\eta}}{\abs{g_1}^2 e^{-\eta}}\, dV_\omega{.}
\]
\end{theorem}

Note that the smooth projective varieties version of \Cref{MAIN THEOREM 2} was proved in \cite{Kim16}. \\
\indent As a corollary, we prove a Siu-type division theorem on compact Kähler manifolds with holomorphic line bundles which admit singular hermitian metrics.
\begin{corollary} \label{Corollary 1}
Let $(X,\omega)$ be a compact Kähler manifold of complex dimension $n\geq 1$. Let $L$ and $H$ be holomorphic line bundles, and let $e^{-\varphi}$ be a semi-positive singular hermitian metric of $H$. Let $k\geq1$ be an integer and fix sections $G_1,\cdots,G_p\in H^0(X, L)$. Define the multiplier ideals 
\vspace{0.2cm}
\[
\mathcal{J}_{k+1}=\mathcal{J}(e^{-\varphi}\abs{G}^{-2(n+k+1)})\,\,\,\text{and}\,\,\,\mathcal{J}_{k}=\mathcal{J}(e^{-\varphi}\abs{G}^{-2(n+k)}){.}
\]
Then
\vspace{0.2cm}
\[
H^0(X,((n+k+1)L+H+K_X)\otimes\mathcal{J}_{k+1})=\sum^p_{j=1}G_j\cdot H^0(X,((n+k)L+H+K_X)\otimes\mathcal{J}_k){.}
\]
\end{corollary}
Note that the projective manifolds version is proved in \cite{Var08}. In \cite[Theorem 1.8.3]{Siu05}, Siu obtained the original version of the above corollary on projective manifolds to prove a type of global generation theorem for multiples of globally generated line bundles. Here we construct a version for compact Kähler manifolds.
\begin{corollary} \label{Corollary 2}
Let $F$ be a holomorphic line bundle over a compact Kähler manifold $(X, \omega)$ of complex dimension $n\geq 1$. Let $a\geq1$ and $b\geq 0$ be integers such that $aF$ and $bF-K_X$ are globally generated over $X$. Then the ring $\bigoplus^\infty_{m= 0}H^0(X, mF)$ is generated by $\bigoplus^{(n+2)a+b-1}_{m= 0}H^0(X, mF)$.
\end{corollary}

Even without the Skoda division theorem, one can show the finite generation of the section ring $\bigoplus^\infty_{m=0}\Gamma(X, mF)$ by inducing a holomorphic surjective fibration from $X$ to some projective variety and concluding the $m$-th Veronese subring of our section ring is finitely generated for some large enough $m$. However, the point here is that the above corollary can provide an effective bound for the degrees of the generators, namely $(n+2)a+b-1$. \\ \\
\indent We briefly explain the sketch of the proof of our main theorem here. \\
\indent Choose $N\in\mathbb{Z}_{>0}$ such that $\alpha N$ is an integer. Then the curvature condition $\sqrt{-1}\partial\xoverline{\partial}\psi\geq\alpha q\cdot \sqrt{-1}\partial\xoverline{\partial}\eta$ implies $e^{-N(\psi-\alpha q\eta)}$ is a singular hermitian metric of the line bundle $NF-\alpha qNE$ with nonnegative curvature. By the Demailly regularization theorem for closed $(1,1)$-currents on a compact Kähler manifold (\textit{cf.} \cite[Theorem 1.1]{Dem94}), we get a family of singular hermitian metrics $\{e^{-\phi_\rho}\}$ of $NF-\alpha qNE$ such that $\phi_\rho$ is smooth on a dense open subset $\Omega_\rho$, which is the complement of a proper analytic subset, and 
\vspace{0.2cm}
\[
\sqrt{-1}\partial\xoverline{\partial} \phi_\rho\geq -N\delta_\rho\omega{,}
\]
where $\lim_{\rho\to 0}\delta_\rho=0$. So $\{e^{-\varphi_\rho}=e^{-((\phi_\rho/N)+\alpha q\log\,\abs{g}^2)}\}$ is a family of singular hermitian metrics of $F$. \\
\indent Next, for each $\rho$ we define a family of metrics $\{\omega_{k,\,\rho}\}^\infty_{k=1}$, a family of smooth functions $\{\kappa_{j,\,k,\,\rho}\}_{j,\,k}$, and two exhaustions $\{\tilde{\Omega}_{k,\,\rho}\}^\infty_{k=1}$ and $\{\Omega_{j,\,k,\,\rho}\}_{j,\,k}$ of $\Omega_\rho$ satisfying
\begin{itemize}
    \item $\omega_{k,\,\rho}$ is a complete Kähler metric defined on $\Omega_\rho$ and $2\omega\geq \omega_{k,\,\rho}$ on $\tilde{\Omega}_{k,\,\rho}$,
    \item $\tilde{\Omega}_{k,\,\rho}\subset \Omega_{j,\,k,\,\rho}$ and $\bigcup_k \tilde{\Omega}_{k,\,\rho}=\bigcup_j\Omega_{j,\,k,\,\rho}=\Omega_\rho$,
    \item $\kappa_{j,\,k,\,\rho}$ has support in $\Omega_{j,\,k,\,\rho}$ and $\kappa_{j,\,k,\,\rho}\equiv 1$ on $\tilde{\Omega}_{k,\,\rho}$,
    \item $\abs{d\kappa_{j,\,k,\,\rho}}_{\omega_{k,\,\rho}}\leq\frac{1}{j}\sup\abs{\chi'}\cdot\abs{dr_{k,\,\rho}}_{\omega_{k,\,\rho}}\leq\frac{A}{j}$.
\end{itemize}
\hspace*{0.8em} We first prove the Skoda division theorem on $\{g_1\neq0\}$ (or we can pick another $g_l$ other than $g_1$). On $\{g_1\neq0\}$, we can regard $E$ as a trivial bundle and define $v=f\wedge\,\xoverline{\partial}(\xoverline{g}/\,\abs{g}^2)$. To prove our main theorem, we aim to solve the $\xoverline{\partial}$-equation $\xoverline{\partial}\xi=v$ with $g\cdot\xi=0$ as a member of 
\vspace{0.2cm}
\[
L^2_{(n,\,0)}(X, F, \tilde{\varphi}, \omega)^{\oplus p}{,}
\]
where $\tilde{\varphi}$ is a smooth hermitian metric we will define during the proof. If we find such $\xi$, letting $h=f(\xoverline{g}/ \abs{g}^2)-\xi$, we can see $\xoverline{\partial}h=0$ and $g\cdot h=f$, which completes the proof. \\
\indent Instead of solving the $\xoverline{\partial}$-equation directly, we solve the perturbed equation
\vspace{0.2cm}
\begin{equation}
(\xoverline{\partial}\xoverline{\partial}^*_{g;\,k,\,\rho}+\xoverline{\partial}^*_{g;\,k,\,\rho}\xoverline{\partial}+C_\rho I)w_{j,\,k,\,\rho}=v
\label{xxx}
\end{equation}
on $\Omega_{j,\,k,\,\rho}$, where $C_\rho$ is chosen to satisfy $\lim_{\rho\to0}C_\rho=0$, and this process heavily depends on the functional analysis theorem \Cref{FA}. Our functional analysis theorem \Cref{FA} is a modified version of the functional analysis theorem which appears in \cite{Che16}. We may refer to \Cref{BKN} to see the definition of the operator $\xoverline{\partial}^*_g$. Using the estimate from \Cref{FA} further, we extract appropriate weak limits $\xi$, $\zeta$, and $\theta$ in 
\vspace{0.2cm}
\[
L^2_{(n,\,0)}(X, F, \tilde{\varphi}, \omega)^{\oplus p}{,}\,\,\,L^2_{(n,\,1)}(X, F, \tilde{\varphi}, \omega)^{\oplus p}{,}\,\,\,\text{and}\,\,\,L^2_{(n,\,1)}(X, F, \tilde{\varphi}, \omega)^{\oplus p}{,}
\]
respectively, starting from the families  
\vspace{0.2cm}
\[
\{\xoverline{\partial}^*_{g;\,k,\,\rho}(w_{j,\,k,\,\rho}\rvert_{\tilde{\Omega}_{k,\,\rho}})\}{,}\,\,\,\{\xoverline{\partial}^*_{g;\,k,\,\rho}\xoverline{\partial}(w_{j,\,k,\,\rho}\rvert_{\tilde{\Omega}_{k,\,\rho}})\}{,}\,\,\,\text{and}\,\,\,\{C_\rho w_{j,\,k,\,\rho}\rvert_{\tilde{\Omega}_{k,\,\rho}}\}{,}
\]
respectively as $j\to\infty$, followed by $k\to\infty$, and finally $\rho\to0$. Then \eqref{xxx} implies $\xoverline{\partial}\xi+\zeta+\theta=v$. Indeed we can prove $\zeta=\theta=0$ as distributions, therefore $\xoverline{\partial}\xi=v$. Since $g\cdot \xoverline{\partial}^*_{g;\,k,\,\rho}(w_{j,\,k,\,\rho}\rvert_{\tilde{\Omega}_{k,\,\rho}})=0$ and this condition is preserved under weak convergence, $g\cdot\xi=0$ also holds, which means $\xi$ is the desired solution. \\
\indent To prove the theorem on the whole $X$, it is enough to apply the Riemann extension theorem by using the $L^2$ finiteness of $\xi$. \\ \\
\indent Our computational strategies are adapted from the methods presented in \cite{Che16}. We explain here some differences between Chen's strategies and ours. \\
\indent In \cite{Che16}, because $\Omega$ is a domain in $\mathbb{C}^n$, it is possible to approximate a plurisubharmonic weight by strictly plurisubharmonic functions using standard convolutions, which leads to strictly positive curvatures. But in our case, we treat a compact Kähler manifold, so we depend on the Demailly regularization theorem for closed positive $(1, 1)$-currents. This leads to not necessarily strictly positive curvatures: $\sqrt{-1}\partial\xoverline{\partial}\phi_\rho\geq -N\delta_\rho\omega$ as explained above. \\
\indent In \cite{Che16}, the curvature considered there is strictly positive, so it is possible to solve the equation $\square_g w_j= v$ directly. But our curvature $\sqrt{-1}\partial\xoverline{\partial}\Psi_\rho$ (\textit{cf.} \Cref{Step3} of the proof of the main theorem) may not be strictly positive. So we need to add an additional strictly positive term, say $C_\rho\omega_{k,\,\rho}$, and solve the perturbed equation 
\vspace{0.2cm}
\[
(\square_{g;\,k,\,\rho}+C_\rho I)w_{j,\,k,\,\rho}= v{.}
\]
\indent \cite{Che16} needs only one parameter to extract required weak limits: $j\to \infty$ to exhaust the open domain $\Omega$ by a sequence of increasing domains $\{\Omega_j\}$. In our case, because we need to regularize the metric $e^{-\phi}$ with $\{e^{-\phi_\rho}\}$ and construct complete metrics $\{\omega_{k,\,\rho}\}$ outside the singular locus $Z_\rho$ (\textit{cf.} \Cref{Step1} of the proof of the main theorem), we need three parameters to extract appropriate weak limits as we explained above: first $j\to \infty$ to exhaust the complete Kähler domain $\Omega_\rho= X\setminus Z_\rho$ with $\{\Omega_{j,\,k,\,\rho}\}$, and $k\to \infty$ to let our metrics $\omega_{k,\,\rho}$ converge back to the original Kähler metric $\omega$, and finally $\rho\to 0$ so that the singular hermitian metrics $\phi_\rho$ converge back to $\phi$. \\
\indent

\subsection*{Acknowledgments}
I would like to express my deep gratitude to Professor Dano Kim for  valuable support and advices.

\vspace{0.5cm}

\section{Preliminaries}

\subsection{Singular hermitian metrics}
Let $L$ be a holomorphic line bundle on a complex manifold $X$. A \textit{singular hermitian metric} on $L$ is a metric which is given in every trivialization $\theta: L\rvert_\Omega\to_{\cong} \Omega\times \mathbb{C}$ by
\vspace{0.2cm}
\[
\abs{\xi}^2_\varphi= \abs{\theta(\xi)}^2 e^{-\varphi(x)}{,}\,\,\,x\in\Omega{,}\,\,\,\xi\in L_x{,}
\]
where $\varphi\in L^1_{\mathrm{loc}}(\Omega)$ is an arbitrary function, called the weight of the metric with respect to the trivialization $\theta$. \\
\indent If $\theta': L\vert_{\Omega'}\to_{\cong} \Omega'\times\mathbb{C}$ is another trivialization, $\varphi'$ the associated weight and $g\in \mathcal{O}^*(\Omega\cap \Omega')$ the transition function, then $\theta'(\xi)=g(x)\theta(\xi)$ for $\xi\in L_x$, and therefore $\varphi'= \varphi+\log\,\abs{g}^2$ on $\Omega\cap \Omega'$. \\
\indent For more details about singular hermitian metrics, we may refer to \cite[3.12]{Dem12}.

\subsection{Global approximation of closed $(1,1)$-currents on a compact Kähler manifold} This subsection introduces a regularization technique for closed $(1,1)$-forms derived from quasi-plurisubharmonic functions. Applying this method to a metric of a holomorphic line bundle with plurisubharmonic weight yields a sequence of approximating metrics that converges to the original metric. These approximating metrics satisfy desirable curvature conditions and are smooth over dense open subsets. 
\vspace{0.2cm}
\begin{theorem}{\cite[Theorem 1.1]{Dem94}} \label{approx}
Let $(X,\omega)$ be a compact complex manifold with a Kähler metric $\omega$.  Assume that 
\setlength{\parskip}{0.5\baselineskip}
\[
T=\tilde{T}+\sqrt{-1} \partial\xoverline{\partial}\sigma
\] 
is a closed $(1,1)$-current on $X$, where $\tilde{T}$ is a smooth real $(1,1)$-form and $\sigma$ is a quasi-plurisubharmonic function. Let $\gamma$ be a continuous real $(1,1)$-form such that $T\geq\gamma$. Then there is a family of closed $(1,1)$-currents 
\setlength{\parskip}{0.5\baselineskip}
\[
T_{\rho}=\tilde{T}+ \sqrt{-1}\partial\xoverline{\partial}\sigma_{\rho}
\] 
on $X$ ($\rho>0$), such that
\begin{itemize}
\item $\sigma_{\rho}$ is a quasi-plurisubharmonic function on $X$, smooth on $X\setminus E_\rho(T)$, increasing with respect to $\rho$, and converges to $\sigma$ on $X$ as $\rho\to0$,
\item $T_{\rho}\geq\gamma-\delta_\rho\omega$, 
\end{itemize}
where $E_\rho(T)$ is the $\rho$-upperlevel set of Lelong numbers, and $\{\delta_\rho\}$ is a family of positive numbers decreasing to 0 as $\rho\to 0$.
\end{theorem}

By \cite{Siu74}, $E_\rho(T)$ is a closed analytic subset of codimension at least one, so $X\setminus E_\rho(T)$ is never empty. \\ 
\indent To apply \Cref{approx} to a singular hermitian metric $e^{-\phi}$ with the curvature condition $\sqrt{-1}\partial\xoverline{\partial}\phi\geq 0$, we must first resolve a technical issue. Since the curvature condition is given only in the sense of currents, the local weights $\phi$ might not be perfectly plurisubharmonic everywhere, that is, they are merely equal to plurisubharmonic functions almost everywhere. 

\begin{lemma} \label{almoste}
Let $X$ be a complex manifold. Let $F$ be a holomorphic line bundle with singular hermitian metric $e^{-\phi'}$. Assume that $\sqrt{-1}\partial\xoverline{\partial}\phi'\geq 0$ in the sense of currents. Then there exists a singular hermitian metric $e^{-\phi}$ such that the weight $\phi$ is plurisubharmonic and $\phi$ equals $\phi'$ almost everywhere.
\end{lemma}

\begin{proof}
Take an open cover $\{U_\alpha\}$ of $X$ such that each $U_\alpha$ is biholomorphic to an open domain in $\mathbb{C}^n$. By \cite[~\rom{3}.1.19]{Dem97}, shrinking $U_\alpha$ if necessary, we may assume there is a plurisubharmonic function $\tilde{\phi}_{\alpha}$ on $U_\alpha$ such that $\sqrt{-1}\partial\xoverline{\partial}\phi'_\alpha=\sqrt{-1}\partial\xoverline{\partial}\tilde{\phi}_\alpha$. By Weyl's lemma, there exists a pluriharmonic function $h_\alpha$ on $U_\alpha$ which equals $\phi'_\alpha-\tilde{\phi}_\alpha$ almost everywhere. In other words, $\phi'_\alpha$ equals the plurisubharmonic function $\tilde{\phi}_\alpha+ h_\alpha$ almost everywhere. \\
\indent We claim the weights $\tilde{\phi}_\alpha+ h_\alpha$ define a singular hermitian line metric of $F$. Let $\{\phi'_\alpha\}$ be the local weights of the metric $e^{-\phi'}$ with respect to the local trivializations on $U_\alpha$. On the overlap $U_\alpha\cap U_\beta$, the local frames are related by $e_\beta=g_{\alpha\beta}e_\alpha$, where $g_{\alpha\beta}\in\mathcal{O}^*_X(U_\alpha\cap U_\beta)$ is a nowhere-vanishing holomorphic transition function, which means
\vspace{0.2cm}
\[
\phi'_\alpha= \phi'_\beta+\log\,\abs{g_{\alpha\beta}}^2{.}
\] 
So the equality
\vspace{0.2cm}
\[
\tilde{\phi}_\alpha+h_\alpha= \phi'_\alpha=\phi'_\beta+\log\,\abs{g_{\alpha\beta}}^2=\tilde{\phi}_\beta+h_\beta+ \log\,\abs{g_{\alpha\beta}}^2
\]
holds almost everywhere on $U_\alpha\cap U_\beta$. But this implies the above equality holds everywhere on $U_\alpha\cap U_\beta$ by the following fact: for any subharmonic function $u$ on a domain, the equality
\vspace{0.2cm}
\[
u(z)=\lim_{r\to 0}\frac{1}{\operatorname{Vol}_{\lambda}B(z, r)}\int_{B(z, r)}u(w)\,d\lambda(w)
\]
holds for all $z$ in the domain. Letting $\phi_\alpha= \tilde{\phi}_\alpha+ h_\alpha$, we finish the proof.
\end{proof}

The next lemma follows from \Cref{almoste} and it fits our situation in \Cref{MAIN THEOREM 1}.

\begin{lemma} \label{074}
Let $X$ be a complex manifold. Let $F$ and $E$ be holomorphic line bundles with singular hermitian metrics $e^{-\psi}$ and $e^{-\tau}$, respectively. Assume that 
\vspace{0.2cm}
\[
\sqrt{-1}\partial\xoverline{\partial}\psi\geq \sqrt{-1}\partial\xoverline{\partial}\tau{.}
\]
Then there exists a singular hermitian metric $e^{-\phi}$ on the holomorphic line bundle $F-E$ such that $\phi$ is plurisubharmonic and equals $\psi- \tau$ almost everywhere. Note that $\psi- \tau$ might not be well-defined on the locus $\{\psi= \pm\infty\}\cup\,\{\tau= \pm\infty\}$, which is of measure zero.
\end{lemma}

\begin{proof}
The difference $\psi-\tau$ of two weights might not be well-defined on the locus $\{\psi= \pm\infty\}\cup\,\{\tau= \pm\infty\}$. So we temporarily define a weight of a singular hermitian metric $\phi'$ on $F-E$ by
\vspace{0.2cm}
\[
\phi'(x)=
\begin{cases}
(\psi-\tau)(x) & x\notin\{\psi= \pm\infty\}\cup\,\{\tau= \pm\infty\} \\
-\infty &  x\in\{\psi= \pm\infty\}\cup\,\{\tau= \pm\infty\}
\end{cases}{.}
\]
Then $\phi'$ is a well-defined weight of a singular hermitian metric, namely $e^{-\phi'}$, and it satisfies $\sqrt{-1}\partial\xoverline{\partial}\phi'\geq 0$ in the sense of currents. By \Cref{almoste}, we conclude there exists a singular hermitian metric $e^{-\phi}$ on $F-E$ such that $\phi$ is plurisubharmonic and equals $\phi'$ almost everywhere, which means it equals $\psi-\tau$ almost everywhere.
\end{proof}

Once we obtain a singular hermitian metric whose weight is plurisubharmonic, we can regularize it with singular hermitian metrics with nice properties by applying \Cref{approx}.

\begin{lemma} \label{regularization}
Let $(X, \omega)$ be a compact Kähler manifold. Let $F$ be a holomorphic line bundle with singular hermitian metric $e^{-\phi}$. Assume that the weight $\phi$ is plurisubharmonic. Then there exists a sequence of singular hermitian metrics $\{e^{-\phi_\rho}\}$ of $F$ and a sequence $\{\delta_\rho\}$ of positive numbers which satisfies the following conditions:
\begin{itemize}
\item $\phi_\rho$ decreases to $\phi$ as $\rho\to 0$,
\item $\delta_\rho$ decreases to $0$ as $\rho\to 0$,
\item $\phi_\rho$ is smooth away from the closed analytic subset $E_\rho(\sqrt{-1}\partial\xoverline{\partial}\phi)$,
\item $\sqrt{-1}\partial\xoverline{\partial}\phi_\rho\geq -\delta_\rho\omega$.
\end{itemize}
\end{lemma}

\begin{proof}
Choose a smooth hermitian metric $e^{-\iota}$ of $F$, and let $\sigma= \phi- \iota$. Then $\sigma$ is a global quasi-plurisubharmonic function on $X$. By \Cref{approx}, there exist a sequence $\{\sigma_\rho\}$ of quasi-plurisubharmonic functions and a sequence $\{\delta_\rho\}$ of positive numbers which satisfy the following conditions:
\begin{itemize}
\item $\sigma_\rho$ decreases to $\sigma$ as $\rho\to 0$,
\item $\delta_\rho$ decreases to $0$ as $\rho\to 0$,
\item $\sigma_\rho$ is smooth away from $E_\rho(\sqrt{-1} \partial\xoverline{\partial}\sigma)=E_\rho(\sqrt{-1}\partial\xoverline{\partial}\phi)$, which is always a proper closed analytic subset by \cite{Siu74},
\item $\sqrt{-1}\partial\xoverline{\partial}(\sigma_\rho+\iota)\geq -\delta_\rho\omega$.
\end{itemize}
Letting $\phi_\rho=\sigma_\rho+\iota$, we get the desired sequence of singular hermitian metrics.
\end{proof}

\vspace{0.5cm}

\section{Bochner-Kodaira-Nakano type inequalities}

\subsection{Classical Bochner-Kodaira-Nakano inequality for $\xoverline{\partial}^*$}
\indent Let $(\Omega,\omega)$ be a relatively compact Kähler domain in a complex manifold. We assume $\omega$ is a Kähler form defined on $\xoverline{\Omega}$. Let $F$ be a holomorphic line bundle with a singular hermitian metric $e^{-\varphi}$ such that $\varphi$ is smooth on $\xoverline{\Omega}$. We define $D_{(n,\,1)}(\Omega, F)$ to be the set of smooth $F$-valued $(n, 1)$-forms with compact support in $\Omega$. Then we have the Bochner-Kodaira-Nakano inequality
\vspace{0.2cm}
\begin{equation}
\norm{\xoverline{\partial} u}^2_\varphi+\norm{\xoverline{\partial}^*u}^2_\varphi\geq \int_\Omega 
([\sqrt{-1}\partial\xoverline{\partial}\varphi, \Lambda_\omega]u, u)_{\varphi}\,dV_\omega{,}
\label{classicBKN}
\end{equation}
where $\xoverline{\partial}^*$ is the formal adjoint operator of $\xoverline{\partial}$. The inequality above originates from Bochner's differential geometric technique, which was adapted to complex manifolds by Kodaira to prove his vanishing and embedding theorems (\textit{cf.} \cite{Kod53}). Nakano later generalized this to holomorphic vector bundles (\textit{cf.} \cite{Nak55}).

\subsection{Bochner-Kodaira-Nakano inequality for $\xoverline{\partial}^*_g$}\label{BKN} In this subsection, we introduce a slightly modified version of the Bochner-Kodaira-Nakano inequality. We need this kind of inequality to solve a $\xoverline{\partial}$-equation, which leads to the solution of our division problem. In this subsection, we follow the contents of \cite{Che16} with slight modifications. \\ \\ 
\indent Let $(\Omega,\omega)$ be a relatively compact Kähler domain in a complex manifold. We assume $\omega$ is a Kähler form defined on $\xoverline{\Omega}$. Let $F$ be a holomorphic line bundle with a singular hermitian metric $e^{-\varphi}$ such that $\varphi$ is smooth on $\xoverline{\Omega}$. Let 
\vspace{0.2cm}
\[ 
g=(g_1,\cdots,g_p)\in \mathcal{O}(\xoverline{\Omega})^{\oplus p}
\]
with $\abs{g}>0$ on $\xoverline{\Omega}$. With these settings, we define $D_{(n,\,k)}(\Omega, F)$ to be the set of smooth $F$-valued $(n, k)$-forms with compact support in $\Omega$, and $L^2_{(n,\,k)}(\Omega, F, \varphi)$ the completion of $D_{(n,\,k)}(\Omega, F)$ with respect to the inner product induced by $\omega$ and $e^{-\varphi}$. We further define
\vspace{0.2cm}
\[
D_k(\Omega, F)=\{u\in D_{(n,\,k)}(\Omega, F)^{\oplus p}: g\cdot u =0\}
\]
and
\vspace{0.2cm}
\[
S_k(\Omega, F,\varphi)=\{u\in L^2_{(n,\,k)}(\Omega, F, \varphi)^{\oplus p}: g\cdot u =0\}{.}
\]
For each $u=(u_1,\cdots,u_p)\in L^2_{(n,\,k)}(\Omega, F, \varphi)^{\oplus p}$, we define
\vspace{0.2cm}
\[
\norm{u}_\varphi=\sqrt{\sum_l \norm{u_l}^2_\varphi}{.}
\]
We can readily check that the completion of $D_k(\Omega, F)$ with respect to the norm $\norm{\boldsymbol{\cdot}}_\varphi$ is contained in $S_k(\Omega, F, \varphi)$. As $g$ is holomorphic, the $\xoverline{\partial}$ operator from $D_{(n,\,0)}(\Omega, F)^{\oplus p}$ to $D_{(n,\,1)}(\Omega, F)^{\oplus p}$ induces a new operator from $D_0(\Omega, F)$ to $D_1(\Omega, F)$, which is still denoted by the same symbol for the sake of simplicity. Let $\xoverline{\partial}^*$ (\textit{resp.} $\xoverline{\partial}^*_g$) denote the formal adjoint of
\vspace{0.2cm} 
\[
\xoverline{\partial}:D_{(n,\,k)}(\Omega, F)^{\oplus p}\to D_{(n,\,k+1)}(\Omega, F)^{\oplus p}\,\,\,(\textit{resp.}\,\xoverline{\partial}:D_k(\Omega, F)\to D_{k+1}(\Omega, F)) 
\]
with respect to the inner product $(\boldsymbol{\cdot}\,,\,\boldsymbol{\cdot})_\varphi$. \\ \\
Now we give a rather explicit formula for $\xoverline{\partial}^*_g$, which is due to Ohsawa (\textit{cf.} \cite{Ohs02a}, \cite[Theorem 4.1]{Che16}). The original version in the literature is slightly different.

\begin{lemma} \label{Ohss} For any $u\in D_1(\Omega, F)$, we have
\vspace{0.2cm}
\[
\xoverline{\partial}^*_gu=\xoverline{\partial}^* u-\xoverline{g}\cdot\sum^p_{l=1}\xoverline{\partial}(\xoverline {g_l}/\,\abs{g}^2)\iprod u_l{,}
\]
where $\iprod$ is the contraction operator and $\xoverline{g}$ (\textit{resp.} $\xoverline{g_l}$) is the complex conjugate of $g$ (\textit{resp.} $g_l$). 
\end{lemma}
\begin{proof}
The orthogonal complement of $S_0(\Omega, F, \varphi)$ in $L^2_{(n,\,0)}(\Omega, F, \varphi)^{\oplus p}$ is
\vspace{0.2cm}
\[
S_0(\Omega, F, \varphi)^\perp=\xoverline{g}\cdot L^2_{(n,\,0)}(\Omega, F, \varphi){.}
\] 
Since $L^2_{(n,\,0)}(\Omega, F, \varphi)$ is a separable Hilbert space and $D_{(n,\,0)}(\Omega, F)$ is dense in $L^2_{(n,\,0)}(\Omega, F, \varphi)$, we may choose by the Gram-Schmidt method a complete orthonormal basis $\{e_j\}\subset\xoverline{g}\cdot D_{(n,\,0)}(\Omega, F)$ of $S_0(\Omega, F, \varphi)^\perp$. Let $P(u)$ be the projection of $\xoverline{\partial}^* u$ to $S_0(\Omega, F, \varphi)$, namely
\vspace{0.2cm}
\[
P(u)=\xoverline{\partial}^* u-\sum_j (\xoverline{\partial}^* u,e_j)_\varphi e_j{.}
\]
Put $e_j=\chi_j\xoverline{g}/\,\abs{g}$. Clearly $\chi_j\in D_{(n,\,0)}(\Omega, F)$ and $\{\chi_j\}$ forms a complete orthonormal basis of $L^2_{(n,\,0)}(\Omega, F, \varphi)$. Since $u\in D_1(\Omega, F)$ and
\vspace{0.2cm}
\[
\xoverline{\partial}e_j=\xoverline{\partial}(\abs{g}\chi_j)\cdot(\xoverline{g}/\,\abs{g}^2)+\abs{g}\chi_j\cdot\xoverline{\partial}(\xoverline{g}/\,\abs{g}^2){,}
\]
it follows that
\vspace{0.2cm}
\[
(\xoverline{\partial}^* u,e_j)_\varphi=(u,\xoverline{\partial}e_j)_\varphi=(u,\abs{g}\chi_j\cdot \xoverline{\partial}(\xoverline{g}/\,\abs{g}^2))_\varphi=\sum^p_{l=1}(\abs{g}\cdot\xoverline{\partial}(\xoverline{g_l}/\,\abs{g}^2)\iprod u_l,\chi_j)_\varphi{.}
\]
Thus
\vspace{0.2cm}
\[
P(u)=\xoverline{\partial}^* u-\sum_j\sum_l(\abs{g}\xoverline{\partial}(\xoverline{g_l}/\,\abs{g}^2)\iprod u_l,\chi_j)_\varphi e_j=\xoverline{\partial}^* u-\xoverline{g}\cdot\sum^p_{l=1}\xoverline{\partial}(\xoverline{g_l}/\,\abs{g}^2)\iprod u_l{,}
\]
which clearly lies in $D_0(\Omega, F)$. Since $(\xoverline{\partial}^* u, w)=(P(u),w)$ for any $w\in D_0(\Omega, F)$, we conclude $\xoverline{\partial}^*_gu=P(u)$.
\end{proof}

\indent Put 
\vspace{0.2cm}
\[
\Phi_g(u)=\xoverline{g}\cdot\sum^p_{l=1}\xoverline{\partial}(\xoverline {g_l}/\,\abs{g}^2)\iprod u_l{.}
\]
By the Cauchy-Schwarz inequality and the Bochner-Kodaira-Nakano inequality (\textit{cf.} \eqref{classicBKN}), for all $u\in D_1(\Omega, F)$ and $\gamma>1$, we have
\vspace{0.2cm}
\[
\frac{\gamma}{\gamma-1}\norm{\xoverline{\partial}^*_gu}^2_\varphi +\norm{\xoverline{\partial}u}^2_\varphi\geq \sum^p_{l=1}\int_\Omega ([\sqrt{-1}\partial\xoverline{\partial}\varphi, \Lambda_\omega]u_l, u_l)_\varphi\,dV_\omega -\gamma\norm{\Phi_g(u)}^2_\varphi{.}
\]
Now assume there is a number $\gamma>1$ such that the right hand side of the previous inequality is no less than
\vspace{0.2cm}
\[
\frac{\gamma}{\gamma-1}\int_\Omega ([\sqrt{-1}\partial\xoverline{\partial}\Psi, \Lambda_\omega]u, u)_\varphi\,dV_\omega= \frac{\gamma}{\gamma-1}  \sum^p_{l=1} \,\int_\Omega ([\sqrt{-1}\partial\xoverline{\partial}\Psi, \Lambda_\omega]u_l, u_l)_\varphi\,dV_\omega{,}
\]
where $\Psi$ is a smooth quasi-plurisubharmonic function on $\xoverline{\Omega}$ such that $\sqrt{-1}\partial\xoverline{\partial}\Psi+t\omega$ is positive on $\xoverline{\Omega}$ for some small $t>0$.
It follows that
\begin{equation}
\begin{split}
& \norm{\xoverline{\partial}^*_gu}^2_\varphi +\norm{\xoverline{\partial}u}^2_\varphi+t\norm{u}^2_\varphi \\
& \geq \int_\Omega ([\sqrt{-1}\partial\xoverline{\partial}\Psi+t\omega, \Lambda_\omega]u, u)_\varphi\,dV_\omega =\sum_l \,\int_\Omega ([\sqrt{-1}\partial\xoverline{\partial}\Psi+t\omega, \Lambda_\omega]u_l, u_l)_\varphi\,dV_\omega
\end{split}
\label{modifiedBoch}
\end{equation}
for all $u\in D_1(\Omega, F)$. The first inequality holds since 
\vspace{0.2cm}
\[
\int_\Omega ([\omega, \Lambda_\omega]u, u)_\varphi\,dV_\omega= \norm{u}^2_\varphi{.}
\]
Put $\square_g=\xoverline{\partial}\xoverline{\partial}^*_g+\xoverline{\partial}^*_g\xoverline{\partial}$. Then we have
\vspace{0.2cm}
\[
\norm{\xoverline{\partial}^*_gu}^2_\varphi +\norm{\xoverline{\partial}u}^2_\varphi+t\norm{u}^2_\varphi=\int_\Omega (\square_g u+tu,u)_\varphi\,dV_\omega
\]
for all $u\in D_1(\Omega, F)$. \\
\indent Let $\mathcal{H}=\mathcal{H}(\Omega, F, \varphi)$ (\textit{resp.} $\mathcal{H}_g=\mathcal{H}_g(\Omega, F, \varphi)$) be the completion of $D_1(\Omega, F)$ in $S_1(\Omega, F, \varphi)$ with respect to the norm $\norm{\boldsymbol{\cdot}}_\varphi$ (\textit{resp.} $(\norm{\xoverline{\partial}^*_g(\boldsymbol{\cdot})}^2_\varphi +\norm{\xoverline{\partial}(\boldsymbol{\cdot})}^2_\varphi+t\norm{\boldsymbol{\cdot}}^2_\varphi)^{1/\,2}$). Clearly $\mathcal{H}_g\subset\mathcal{H}$, and we may prove the following theorem.

\begin{theorem}\label{FA}
For any $v\in\mathcal{H}$, there is a unique weak solution $w\in\mathcal{H}_g$ of the equation $\square_g w+tw=v$ such that
\begin{flalign*}
&\max\big\{\int_\Omega ([\sqrt{-1}\partial\xoverline{\partial}\Psi+t\omega, \Lambda_\omega]w, w)_\varphi\,dV_\omega, \norm{\xoverline{\partial}^*_gw}^2_\varphi, \norm{\xoverline{\partial}w}^2_\varphi, t\norm{w}^2_\varphi\big\} \\
&\leq \int_\Omega ([\sqrt{-1}\partial\xoverline{\partial}\Psi+t\omega, \Lambda_\omega]^{-1}v, v)_\varphi\,dV_\omega= \norm{v}^2_{\varphi, \sqrt{-1}\partial\bar{\partial}\Psi +t\omega}{.}
\end{flalign*}
\end{theorem}
Recall that $t>0$ is chosen so that $\sqrt{-1}\partial\xoverline{\partial}\Psi+t\omega>0$. The term $\norm{v}^2_{\varphi, \sqrt{-1}\partial\bar{\partial}\Psi +t\omega}$ means the square of the norm of $v$ with respect to $\sqrt{-1}\partial\xoverline{\partial}\Psi+t\omega$ and $\varphi$.
\begin{proof}
Define the linear functional 
\vspace{0.2cm}
\[
F(x)=\int_\Omega (x, v)_\varphi\,dV_\omega
\]
for all $x\in\mathcal{H}_g$. Then by the Cauchy-Schwarz inequality,
\vspace{0.2cm}
\begin{flalign*}
&\abs{F(x)}^2\leq \Bigl(\int_\Omega ([\sqrt{-1}\partial\xoverline{\partial}\Psi+t\omega, \Lambda_\omega]x, x)_\varphi\,dV_\omega\Bigr)\Bigl(\int_\Omega ([\sqrt{-1}\partial\xoverline{\partial}\Psi+t\omega, \Lambda_\omega]^{-1}v, v)_\varphi\,dV_\omega\Bigr) \\
&\leq (\norm{v}^2_{\varphi, \sqrt{-1}\partial\bar{\partial}\Psi +t\omega})(\norm{\xoverline{\partial}^*_gx}^2_\varphi +\norm{\xoverline{\partial}x}^2_\varphi+t\norm{x}^2_\varphi){,} \\
\end{flalign*}
and there exists a unique $w\in\mathcal{H}_g$ such that 
\vspace{0.2cm}
\[
\int_\Omega (\xoverline{\partial}^*_gx, \xoverline{\partial}^*_gw)_\varphi\,dV_\omega +\int_\Omega(\xoverline{\partial}x, \xoverline{\partial}w)_\varphi\,dV_\omega+t\int_\Omega (x, w)_\varphi\,dV_\omega=\int_\Omega (x, v)_\varphi\,dV_\omega
\]
by the Riesz representation theorem. The Riesz representation theorem also implies that $\norm{\xoverline{\partial}^*_gw}^2_\varphi +\norm{\xoverline{\partial}w}^2_\varphi+t\norm{w}^2_\varphi\leq\norm{v}^2_{\varphi, \sqrt{-1}\partial\bar{\partial}\Psi +t\omega}$, which means 
\vspace{0.2cm}
\[
\int_\Omega ([\sqrt{-1}\partial\xoverline{\partial}\Psi+t\omega, \Lambda_\omega]w, w)_\varphi\,dV_\omega\leq\norm{v}^2_{\varphi, \sqrt{-1}\partial\bar{\partial}\Psi +t\omega}{.}
\]
\end{proof}

Before moving on to the next section, we introduce several lemmas that are helpful for handling the Bochner-Kodaira-Nakano inequalities more effectively.

\begin{lemma}[\cite{Sko72}] \label{Skoda}
Let $g=(g_1,\cdots,g_p)$ be holomorphic functions on a domain $U\subset\mathbb{C}^n$, and let $q=\min\{n,p-1\}$. Then for any matrix $\varrho$ of size $p\times n$, in Einstein summation convention for $\nu$ and $\mu$,
\setlength{\parskip}{0.5\baselineskip}
\[
q\sum^p_{l=1}(\varrho^\nu_l\xoverline{\varrho^\mu_l}\partial_\nu\partial_{\bar{\mu}}\log\,\abs{g}^2)\geq\abs{g}^2 \Bigl| \sum\nolimits^p_{l=1}\varrho^\nu_l\partial_\nu(g_l/\,\abs{g}^2) \Bigr|^2{.}
\]
\end{lemma}
For the proof, we refer to \cite{Sko72}. With \Cref{Skoda}, we prove a lemma which is used in the proof of the main theorem.

\begin{lemma}\label{726}
Let $g$, $U$, and $q$ be as in \Cref{Skoda}. Then for every positive definite hermitian form $\omega$ and $\epsilon> 0$, we have
\[
\abs{\xoverline{\partial}(\xoverline{g}/\, \abs{g}^2)}^2_{\sqrt{-1}\partial\bar{\partial}\log\,\abs{g}^2+ \epsilon\omega}\leq \frac{q}{\abs{g}^2}
\]
on $\{g_1\neq 0\}$ (or we may pick another $g_l$ other than $g_1$).
\end{lemma}

\begin{proof}
Put 
\vspace{0.2cm}
\[
H_{\mu\bar{\nu}}= \frac{\partial^2 \log\,\abs{g}^2}{\partial z_\mu \partial\bar{z}_\nu}{.}
\]
Since $\log\,\abs{g}^2$ is plurisubharmonic, the hermitian matrix $H= (H_{\mu\bar{\nu}})$ is positive semi-definite. \\
\indent Now set
\vspace{0.2cm}
\[
H_{\epsilon,\, \mu\bar{\nu}}= H_{\mu\bar{\nu}}+ \epsilon\omega_{\mu\bar{\nu}}{.}
\]
Let $H^{\bar{\nu}\mu}_\epsilon$ denote the inverse matrix of $H_\epsilon$. Define
\vspace{0.2cm}
\[
\eta_{l,\,\mu}=\frac{\partial}{\partial z_\mu}\left(\frac{g_l}{\abs{g}^2}\right)
\]
and
\vspace{0.2cm}
\[
\varrho^\mu_l= \sum_\lambda H^{\bar{\lambda}\mu}_\epsilon \xoverline{\eta_{l,\,\lambda}}{.}
\]
Also put
\vspace{0.2cm}
\[
S_\epsilon= \sum_{l,\,\mu,\,\nu} H^{\bar{\nu}\mu}_\epsilon \eta_{l,\,\mu}\xoverline{\eta_{l,\,\nu}}= \sum_{l,\,\mu,\,\nu}H_{\epsilon,\,\mu\bar{\nu}}\varrho^\mu_l\xoverline{\varrho^\nu_l}= \abs{\xoverline{\partial}(\xoverline{g}/\, \abs{g}^2)}^2_{\sqrt{-1}\partial\bar{\partial}\log\,\abs{g}^2+ \epsilon\omega}{.}
\]
Applying \Cref{Skoda} to our choice of $\varrho$ and $\eta$ and using the fact that $H\leq H_\epsilon$, we obtain
\vspace{0.2cm}
\[
\abs{g}^2S^2_\epsilon= \abs{g}^2\abs{\sum_{l,\,\mu}\varrho^\mu_l \eta_{l,\,\mu}}^2\leq q\sum_{l,\,\mu,\,\nu}H_{\mu\bar{\nu}}\varrho^\mu_l\xoverline{\varrho^\nu_l}\leq q\sum_{l,\,\mu,\,\nu}H_{\epsilon,\,\mu\bar{\nu}}\varrho^\mu_l\xoverline{\varrho^\nu_l}= qS_\epsilon{,}
\]
which gives the desired inequality.
\end{proof}

\vspace{0.5cm}

\section{Proof of \Cref{MAIN THEOREM 1}}
We prove \Cref{MAIN THEOREM 1} in this section. Assume that $g= (g_1, \cdots, g_p)$ is not identically zero. After relabeling, we may assume $g_1\not\equiv 0$. First, we prove the case $q=0$. In this case, $p=1$, so the equation we must solve is simply $h_1g_1= f$. On the open dense subset $X\setminus\,\{g_1= 0\}$, the solution is uniquely determined as $h_1=f/\, g_1$, and the given condition gives
\vspace{0.2cm}
\[
\int_X \frac{\abs{f}_\omega^2e^{-\psi-\eta}}{\abs{g_1}^2e^{-\eta}}\,dV_\omega= \int_{X\setminus \{g_1= 0\}} \abs{h_1}^2_\omega e^{-\psi}<\infty{.}
\]
This implies $h_1$ is locally square-integrable across the locus $\{g_1= 0\}$ since we may assume $\psi$ is plurisubharmonic by \Cref{almoste}. By the Riemann extension theorem, $h_1$ extends to a global holomorphic section, which completes the proof for the case $q= 0$. \\ \\

To prove the case $q\neq 0$, we divide the proof into a few steps. From now on, we cannot assume $q=0$ due to the term $2(\alpha+1)/\,q(\alpha-1)^2$ in \Cref{Step4}.

\begin{step}[Approximation of singular hermitian metrics of the line bundle $NF-\alpha qNE$ for some large $N$] \label{Step1}
Choose a positive integer $N$ such that $\alpha N$ is a positive integer. By the curvature condition $\sqrt{-1}\partial\xoverline{\partial}\psi\geq \alpha q\cdot \sqrt{-1}\partial\xoverline{\partial}\eta$ and \Cref{074}, there exists a well-defined singular hermitian metric $e^{-\phi}$ of $NF-\alpha qNE$ such that the weight $\phi$ is plurisubharmonic and $\phi$ equals $N(\psi-\alpha q\eta)$ almost everywhere. By \Cref{regularization}, there exists a sequence of singular hermitian metrics $\{e^{-\phi_\rho}\}$ of $NF- \alpha qN E$ and a sequence $\{\delta_\rho\}$ of positive numbers which satisfies the following conditions:
\begin{itemize}
\item $\phi_\rho$ decreases to $\phi$ as $\rho\to 0$,
\item $\delta_\rho$ decreases to $0$ as $\rho\to 0$,
\item $\phi_\rho$ is smooth away from the closed analytic subset $E_\rho(\sqrt{-1}\partial\xoverline{\partial}\phi)$,
\item $\sqrt{-1}\partial\xoverline{\partial}\phi_\rho\geq -N\delta_\rho\omega$.
\end{itemize}   
At this point, we get a new family of singular hermitian metrics 
\setlength{\parskip}{0.5\baselineskip}
\[
\{e^{-\varphi_\rho}=e^{-(\frac{\phi_\rho}{N}+\alpha q \log\,\abs{g}^2)}\}
\]
of $F$. \\
\indent Let
\vspace{0.2cm}
\[
Z_\rho\coloneq E_\rho(\sqrt{-1} \partial\xoverline{\partial}(N(\psi- \alpha q \eta)))\cup\{g_1=0\}= E_\rho(\sqrt{-1} \partial\xoverline{\partial}\phi)\cup\{g_1=0\}{,}
\]
then by the definition of $Z_\rho$, we can regard $E$ as a trivial line bundle on $\Omega_\rho=X\setminus Z_\rho$ and $g_l$ as a holomorphic function on $\Omega_\rho$, so that we have no difficulty in applying \Cref{Ohss} or \Cref{Skoda} later. When we define $Z_\rho$, we can pick another $g_l$ other than $g_1$. \\
\indent Before we pass to the next step, we take a $\rho^*>0$ and a smooth hermitian metric $\tilde{\varphi}$ such that $\tilde{\varphi}\geq\varphi_{\rho^*}$. Such $\tilde{\varphi}$ exists because $\varphi_\rho$ is locally bounded above. These will be used in \Cref{Step5}.
\end{step}

\vspace{0.5cm}

\begin{step}[Construction of a family of metrics $\{\omega_{k,\,\rho}\}$, a family of smooth functions $\{\kappa_{j,\,k,\,\rho}\}$, and two sequences of subdomains $\{\tilde{\Omega}_{k,\,\rho}\}$ and $\{\Omega_{j,\,k,\,\rho}\}$] \label{Step2}
Let $\omega_\rho$ be a complete Kähler metric on $\Omega_\rho=X\setminus Z_\rho$ which satisfies the following condition: for any $K>1$, the open subset $\{ x\in \Omega_\rho= X\setminus Z_\rho: K\omega > \omega_\rho \}$ is relatively compact in $\Omega_\rho$. Such a metric always exists by the following lemma.
\begin{lemma}
Let $(X, \omega)$ be a compact Kähler manifold, and $Z$ be a proper analytic subset. Then there exists a complete Kähler metric $\omega'$ on $X\setminus Z$ which satisfies the following condition: for any $K>1$, the open subset $\Omega_K= \{ x\in X\setminus Z: K\omega > \omega' \}$ is relatively compact in $X\setminus Z$.
\end{lemma}
\begin{proof}
By Hironaka's theorem (\textit{cf.} \cite{Hir64}), there exists a proper modification $\pi:\tilde{X}\to X$ such that $\tilde{X}$ is also a compact Kähler manifold and the preimage $\pi^{-1}(Z)$ is an SNC divisor whose support is equal to $\sum^N_{j=1}E_j$. Here each $E_j$ is a prime divisor. Take a Kähler metric $\tilde{\omega}$ on $\tilde{X}$. Choose a smooth hermitian metric $h_j$ on the line bundle $\mathcal{O}(E_j)$ and canonical section $s_j$ cutting out $E_j$. By scaling the metric, we may assume $\abs{s_j}^2_{h_j}<e^{-1}$. \\
\indent Now on $\tilde{X}\setminus\pi^{-1}(Z)$, define the complete Kähler metric
\vspace{0.2cm}
\[
\tilde{\omega}'=C\tilde{\omega}-\sum^N_{j=1}\sqrt{-1}\partial\xoverline{\partial}\log(-\log\,\abs{s_j}^2_{h_j})
\]
for a sufficiently large $C>0$. We pushforward this metric to $X\setminus Z$ and call it $\omega'$. We claim $\omega'$ satisfies the condition: for any $K>1$, the open subset $\Omega_K= \{ x\in X\setminus Z: K\omega > \omega' \}$ is relatively compact in $X\setminus Z$. Pulling back $\Omega_K$ to $\tilde{X}\setminus\pi^{-1}(Z)$, it is equivalent to consider the condition $K\pi^*\omega>\tilde{\omega}'$. By compactness of $\tilde{X}$, there exists $A>0$ such that $A\tilde{\omega}\geq\pi^*\omega$. So it is enough to prove that the set $\{x\in \tilde{X}\setminus\pi^{-1}(Z): A\cdot K\cdot \tilde{\omega}>\tilde{\omega}'\}$ is relatively compact in $\tilde{X}\setminus\pi^{-1}(Z)$, which is clear.
\end{proof}
Define
\vspace{0.2cm}
\[
\omega_{k,\,\rho}=\omega+\frac{1}{k}\omega_\rho{,}
\]
which is still a complete Kähler metric. Define a sequence of domains 
\vspace{0.2cm}
\[
\tilde{\Omega}_{k,\,\rho}=\{x\in \Omega_\rho: \omega_\rho(x)<k\omega(x)\}\,\,\,(k\in\mathbb{Z}_{>0}){.}
\]
Then on $\tilde{\Omega}_{k,\,\rho}$ the inequality $\omega_{k,\,\rho}\leq2\omega$ holds, and
\vspace{0.2cm}
\[
\bigcup_k \tilde{\Omega}_{k,\,\rho}=\Omega_\rho{.}
\] 
\indent Take a smooth exhaustion function $r_{k,\,\rho}$ on $\Omega_\rho$ such that $\abs{d r_{k,\,\rho}}_{\omega_{k,\,\rho}}\leq 1$. Such an exhaustion function always exists due to \cite[~\rom{8}.2.4]{Dem97}. Since $\tilde{\Omega}_{k,\,\rho}$ is relatively compact in $\Omega_\rho$, there exists $R_{k,\,\rho}>0$ such that $\tilde{\Omega}_{k,\,\rho}\subset\{x: r_{k,\,\rho}(x)<R_{k,\,\rho}\}$. Now define again a sequence of domains
\vspace{0.2cm}
\[
\Omega_{j,\,k,\,\rho}=\{x\in \Omega_\rho:r_{k,\,\rho}(x)<R_{k,\,\rho}+j+1\}
\]
and a smooth function $\chi:\mathbb{R}\to[0,1]$ such that $\chi\rvert_{(-\infty, 0]} \equiv 1$ and $\chi\rvert_{[1, \infty)} \equiv 0$. We let $A=\sup\abs{\chi'}<\infty$. With the exhaustion function $r_{k,\,\rho}$, we define a cut-off function
\vspace{0.2cm}
\[
\kappa_{j,\,k,\,\rho}(x)=\chi(\frac{r_{k,\,\rho}(x)-R_{k,\,\rho}}{j}){.}
\]
Then we can observe the following properties:
\begin{itemize}
\item $2\omega\geq \omega_{k,\,\rho}$ on $\tilde{\Omega}_{k,\,\rho}$,
\item $\kappa_{j,\,k,\,\rho}$ has support in $\Omega_{j,\,k,\,\rho}$,
\item $\kappa_{j,\,k,\,\rho}\equiv 1$ on $\tilde{\Omega}_{k,\,\rho}$,
\item $\abs{d\kappa_{j,\,k,\,\rho}}_{\omega_{k,\,\rho}}\leq\frac{1}{j}\sup\abs{\chi'}\cdot\abs{dr_{k,\,\rho}}_{\omega_{k,\,\rho}}\leq\frac{A}{j}$.
\end{itemize}
\end{step}

\vspace{0.5cm}

\begin{step}[Obtaining a Bochner-Kodaira-Nakano inequality on $\Omega_{j,\,k,\,\rho}$ with respect to $\varphi_\rho$ and $\omega_{k,\,\rho}$] \label{Step3}
Take $\gamma=(\alpha+1)/2$, $t=C_\rho=2((\alpha-1)/(\alpha+1))\delta_\rho$, and
\vspace{0.2cm}
\[
\Psi=\Psi_\rho=\frac{\alpha-1}{\alpha+1}(\varphi_\rho-\frac{\alpha+1}{2}q\log\,\abs{g}^2)=\frac{\alpha-1}{\alpha+1}(\frac{\phi_\rho}{N}+\frac{\alpha-1}{2}q\log\,\abs{g}^2){.}
\]
Applying our $\gamma$, $t$, and $\Psi$ to the situation of \Cref{BKN}, combining with \Cref{Skoda}, we have a Bochner-Kodaira-Nakano inequality for $\xoverline{\partial}^*_{g;\,k,\,\rho}$ as an operator on $\Omega_{j,\,k,\,\rho}$ with respect to $e^{-\varphi_\rho}$ and the Kähler metric $\omega_{k,\,\rho}$: 
\vspace{0.2cm}
\begin{equation}
\begin{split}
& \norm{\xoverline{\partial}^*_{g;\,k,\,\rho}u}^2_{\varphi_\rho, \omega_{k,\,\rho}} +\norm{\xoverline{\partial}u}^2_{\varphi_\rho, \omega_{k,\,\rho}}+C_\rho\norm{u}^2_{\varphi_\rho, \omega_{k,\,\rho}}\geq \int_{\Omega_{j,\,k,\,\rho}}([\sqrt{-1}\partial\xoverline{\partial}\Psi_\rho+C_\rho\omega_{k,\,\rho}, \Lambda_{\omega_{k,\,\rho}}]u, u)_{\varphi_\rho,\,\omega_{k,\,\rho}}\,dV_{\omega_{k,\,\rho}} \\
& =\int_{\Omega_{j,\,k,\,\rho}}([\sqrt{-1}\partial\xoverline{\partial}\Psi_\rho+\frac{2(\alpha-1)}{\alpha+1}\delta_\rho\omega_{k,\,\rho}, \Lambda_{\omega_{k,\,\rho}}]u, u)_{\varphi_\rho,\,\omega_{k,\,\rho}}\,dV_{\omega_{k,\,\rho}}{,}
\end{split}
\label{newBoc}
\end{equation}
which corresponds to \eqref{modifiedBoch}, for all $u\in D_1(\Omega_{j,\,k,\,\rho}, F)$. Here we use the notation $\xoverline{\partial}^*_{g;\,k,\,\rho}$ to specify the metrics $\omega_{k,\,\rho}$ and $e^{-\varphi_\rho}$ we are using. Note that 
\vspace{0.2cm}
\begin{equation}
\sqrt{-1}\partial\xoverline{\partial}\Psi_\rho+\frac{2(\alpha-1)}{\alpha+1}\delta_\rho\omega_{k,\,\rho}\geq\frac{(\alpha-1)^2}{2(\alpha+1)}q\sqrt{-1}\partial\xoverline{\partial}\log\,\abs{g}^2+ \frac{\alpha- 1}{\alpha+ 1}\delta_\rho \omega_{k,\,\rho}> 0{.}
\label{ttt}
\end{equation}
\end{step}
\vspace{0.5cm}

\begin{step}[Extracting a unique weak solution $w_{j,\,k,\,\rho}$ of a $(\square_{g;\,k,\,\rho}+C_\rho I)$-equation] \label{Step4}
Put $v=f\wedge\,\xoverline{\partial}(\xoverline{g}/\,\abs{g}^2)$ on $X\setminus\{g_1=0\}$, which is a member of $L^2_{(n,\,1)}(\Omega_{j,\,k,\,\rho}, F, \varphi_\rho, \omega_{k,\,\rho})^{\oplus p}$. Since $g\cdot v=0$ and $g\cdot (\kappa v)=0$ for any $\kappa\in C^\infty_0(\Omega_{j,\,k,\,\rho})$, it follows that $v\in\mathcal{H}(\Omega_{j,\,k,\,\rho}, F, \varphi_\rho, \omega_{k,\,\rho})$. We refer to \Cref{BKN} to check the definition of $\mathcal{H}(\Omega_{j,\,k,\,\rho}, F, \varphi_\rho, \omega_{k,\,\rho})$. Thus by \Cref{FA}, there is a unique $w_{j,\,k,\,\rho}\in\mathcal{H}_g(\Omega_{j,\,k,\,\rho}, F, \varphi_\rho, \omega_{k,\,\rho})$ such that 
\vspace{0.2cm}
\[
(\xoverline{\partial}\xoverline{\partial}^*_{g;\,k,\,\rho}+\xoverline{\partial}^*_{g;\,k,\,\rho}\xoverline{\partial}+C_\rho I)w_{j,\,k,\,\rho}=(\square_{g;\,k,\,\rho}+C_\rho I)w_{j,\,k,\,\rho}=v
\]
on $\Omega_{j,\,k,\,\rho}$ and
\vspace{0.2cm}
\begin{flalign*}
&\max \Bigl\{  \norm{\xoverline{\partial} w_{j,\,k,\,\rho}}^2_{\varphi_\rho,\,\omega_{k,\,\rho}}, \, 
               \norm{ \xoverline{\partial}^*_{g;\,k,\,\rho} w_{j,\,k,\,\rho}}^2_{\varphi_\rho,\, \omega_{k,\,\rho}}, \, 
               C_\rho \norm{ w_{j,\,k,\,\rho}}^2_{\varphi_\rho, \,\omega_{k,\,\rho}} \Bigr\} \\
    &\leq \norm{v}^2_{\varphi_\rho,\, \sqrt{-1}\partial\bar{\partial}\Psi_\rho+2((\alpha-1)/(\alpha+1))\delta_\rho\omega_{k,\,\rho}}\leq \frac{2(\alpha+1)}{q(\alpha-1)^2} \int_{\Omega_{j,\,k,\,\rho}} \abs{f}_\omega^2 \abs{\xoverline{\partial}(\xoverline{g}/\,\abs{g}^2)}^2_{\sqrt{-1}\partial\bar{\partial}\log\,\abs{g}^2+ (2\delta_\rho/\, q(\alpha- 1))\omega_{k,\,\rho}} e^{-\varphi_\rho}\,dV_{\omega} \\
&\leq \frac{2(\alpha+1)}{(\alpha-1)^2} \int_{\Omega_{j,\,k,\,\rho}} \abs{f}_\omega^2 \abs{g}^{-2}e^{-\varphi_\rho}\,dV_\omega \leq \frac{2(\alpha+1)}{(\alpha-1)^2} \int_{\Omega_{j,\,k,\,\rho}} \abs{f}_\omega^2 \abs{g}^{-2(\alpha q+1)} e^{-\frac{\phi_\rho}{N}}\,dV_\omega \\
&\leq \frac{2(\alpha+1)}{(\alpha-1)^2} \int_X \abs{f}_\omega^2 \abs{g}^{-2(\alpha q+1)} e^{-\frac{\phi}{N}}\,dV_\omega= \frac{2(\alpha+1)}{(\alpha-1)^2} \int_X \frac{\abs{f}_\omega^2 e^{-\psi-\eta}}{(\abs{g}^2 e^{-\eta})^{\alpha q+1}}\,dV_\omega= M < \infty{,}
\end{flalign*}
where the second inequality follows from \eqref{ttt} and the third inequality follows from \Cref{726}. Here $\abs{g}$ is just the absolute value of $g$ as a tuple of holomorphic functions. Recall that we can regard $g_l$ as a holomorphic function on $\Omega_\rho$. Since $(\square_{g;\,k,\,\rho}+C_\rho I)w_{j,\,k,\,\rho}=v$ and $\square_{g;\,k,\,\rho}+C_\rho I$ is an elliptic operator of order two, we know $w_{j,\,k,\,\rho}$ is smooth as a form on $\Omega_{j,\,k,\,\rho}$. 
\end{step}
\vspace{0.5cm}

\begin{step}[Extracting weak limits $\xi$, $\zeta$, and $\theta$ as $j\to\infty$, followed by $k\to\infty$, and finally $\rho\to0$] \label{Step5}
From now on, we consider mainly the family $\{w_{j,\,k,\,\rho}\}$ with $\rho\leq\rho^*$, so that $\tilde{\varphi}\geq \varphi_\rho$. Recall that $\rho^*$ was chosen in \Cref{Step1}. Depending on the context, we view $w_{j,\,k,\,\rho}$ as a function on either $\Omega_{j,\,k,\,\rho}$ or its subset $\tilde{\Omega}_{k,\,\rho}$ via restriction. In the latter case, we denote it as $w_{j,\,k,\,\rho}\rvert_{\tilde{\Omega}_{k,\,\rho}}$. \\
\indent The inequality 
\vspace{0.2cm}
\begin{flalign*}
&\norm{\sqrt{C_\rho}w_{j,\,k,\,\rho}\rvert_{\tilde{\Omega}_{k,\,\rho}}}^2_{\tilde{\varphi},\,\omega}=\int_{\tilde{\Omega}_{k,\,\rho}}\abs{\sqrt{C_\rho}w_{j,\,k,\,\rho}}^2_{\omega}e^{-\tilde{\varphi}}\,dV_\omega \\
&\leq 2\int_{\tilde{\Omega}_{k,\,\rho}}\abs{\sqrt{C_\rho}w_{j,\,k,\,\rho}}^2_{\omega_{k,\,\rho}}e^{-\varphi_\rho}\,dV_{\omega_{k,\,\rho}}=2\norm{\sqrt{C_\rho}w_{j,\,k,\,\rho}\rvert_{\tilde{\Omega}_{k,\,\rho}}}^2_{\varphi_\rho,\,\omega_{k,\,\rho}}\leq2\norm{\sqrt{C_\rho}w_{j,\,k,\,\rho}}^2_{\varphi_\rho,\,\omega_{k,\,\rho}}\leq 2\cdot M
\end{flalign*}
implies that for each pair $(k, \rho)$, the family $\{\sqrt{C_\rho}w_{j,\,k,\,\rho}\rvert_{\tilde{\Omega}_{k,\,\rho}}\}_j$ has a weakly convergent subsequence in 
\vspace{0.2cm}
\[
L^2_{(n,\,1)}(\tilde{\Omega}_{k,\,\rho}, F, \tilde{\varphi}, \omega)^{\oplus p}{.}
\]
	Let us call the weak limit $\sqrt{C_\rho}w_{k,\,\rho}\rvert_{\tilde{\Omega}_{k,\,\rho}}$. But since $\norm{\sqrt{C_\rho}w_{k,\,\rho}\rvert_{\tilde{\Omega}_{k,\,\rho}}}^2_{\tilde{\varphi},\,\omega}\leq2\cdot M$, the family $\{\sqrt{C_\rho}w_{k,\,\rho}\rvert_{\tilde{\Omega}_{k,\,\rho}}\}_k$ has a weakly convergent subsequence in
	\vspace{0.2cm} 
	\[
	L^2_{(n,\,1)}(X, F, \tilde{\varphi}, \omega)^{\oplus p}{.}
	\]
Let us call the weak limit $\sqrt{C_\rho}w_\rho$ with the estimate $\norm{\sqrt{C_\rho}w_\rho}^2_{\tilde{\varphi},\,\omega}\leq 2\cdot M$, which implies $\norm{C_\rho w_\rho}^2_{\tilde{\varphi},\,\omega}\leq 2\cdot C_\rho\cdot M$.    \\
\indent Similarly the inequality 
\vspace{0.2cm}
\begin{flalign*}
& \norm{\xoverline{\partial}^*_{g;\, k,\,\rho}w_{j,\,k,\,\rho}\rvert_{\tilde{\Omega}_{k,\,\rho}}}^2_{\tilde{\varphi},\,\omega}= \int_{\tilde{\Omega}_{k,\,\rho}} \abs{\xoverline{\partial}^*_{g;\, k,\,\rho}w_{j,\,k,\,\rho}}^2_\omega e^{-\tilde{\varphi}}\,dV_\omega
\leq \int_{\tilde{\Omega}_{k,\,\rho}} \abs{\xoverline{\partial}^*_{g;\, k,\,\rho}w_{j,\,k,\,\rho}}^2_{\omega_{k,\, \rho}} e^{-\varphi_{\rho}}\,dV_{\omega_{k,\, \rho}}\\
& = \norm{\xoverline{\partial}^*_{g;\, k,\,\rho}w_{j,\,k,\,\rho}\rvert_{\tilde{\Omega}_{k,\,\rho}}}^2_{\varphi_{\rho},\,\omega_{k,\,\rho}}\leq \norm{\xoverline{\partial}^*_{g;\, k,\,\rho}w_{j,\,k,\,\rho}}^2_{\varphi_{\rho},\, \omega_{k,\,\rho}}\leq M{.}
\end{flalign*}
implies that for each pair $(k, \rho)$, the family $\{\xoverline{\partial}^*_{g;\, k,\,\rho}w_{j,\,k,\,\rho}\rvert_{\tilde{\Omega}_{k,\,\rho}}\}_j$ has a weakly convergent subsequence in
\vspace{0.2cm}
\[
L^2_{(n,\, 0)}(\tilde{\Omega}_{k,\,\rho}, F, \tilde{\varphi}, \omega)^{\oplus p}{.}
\]
Let us call the weak limit $\xi_{k,\,\rho}\rvert_{\tilde{\Omega}_{k,\,\rho}}$. But since $\norm{\xi_{k,\,\rho}\rvert_{\tilde{\Omega}_{k,\,\rho}}}^2_{\tilde{\varphi},\,\omega}\leq M$, the family $\{\xi_{k,\,\rho}\rvert_{\tilde{\Omega}_{k,\,\rho}}\}_k$ has a weakly convergent subsequence in
\vspace{0.2cm}
\[
L^2_{(n,\, 0)}(X, F, \tilde{\varphi}, \omega)^{\oplus p}{.}
\]
Let us call the weak limit $\xi_{\rho}$ with the estimate $\norm{\xi_{\rho}}^2_{\tilde{\varphi},\,\omega}\leq M$. \\
\indent Now for each $\rho$, we want to extract an appropriate weak limit $\zeta_\rho$ from the family $\{\xoverline{\partial}^*_{g;\,k,\,\rho}\xoverline{\partial}(w_{j,\,k,\,\rho}\rvert_{\tilde{\Omega}_{k,\,\rho}})\}_{j,\,k}$. For this, we first note that
\vspace{0.2cm}
\begin{flalign*}
&(-C_\rho\xoverline{\partial}w_{j,\,k,\,\rho},\kappa^2_{j,\,k,\,\rho}\xoverline{\partial}w_{j,\,k,\,\rho})_{\varphi_\rho,\,\omega_{k,\,\rho}}=(\xoverline{\partial}\xoverline{\partial}^*_{g;\,k,\,\rho}\xoverline{\partial}w_{j,\,k,\,\rho}, \kappa^2_{j,\,k,\,\rho}\xoverline{\partial}w_{j,\,k,\,\rho})_{\varphi_\rho,\,\omega_{k,\,\rho}} \\
&=\norm{\kappa_{j,\,k,\,\rho}\xoverline{\partial}^*_{g;\,k,\,\rho}\xoverline{\partial}w_{j,\,k,\,\rho}}^2_{\varphi_\rho,\,\omega_{k,\,\rho}}-2(\kappa_{j,\,k,\,\rho}\xoverline{\partial}\kappa_{j,\,k,\,\rho}\wedge\xoverline{\partial}^*_{g;\,k,\,\rho}\xoverline{\partial}w_{j,\,k,\,\rho},\xoverline{\partial}w_{j,\,k,\,\rho})_{\varphi_\rho,\,\omega_{k,\,\rho}}{,}
\end{flalign*}
where the smoothness of $w_{j,\,k,\,\rho}$ is necessarily used in the second equality, so that
\vspace{0.2cm}
\begin{flalign*}
&\norm{\xoverline{\partial}^*_{g;\,k,\,\rho}\xoverline{\partial}(w_{j,\,k,\,\rho}\rvert_{\tilde{\Omega}_{k,\,\rho}})}^2_{\tilde{\varphi},\,\omega}=\int_{\tilde{\Omega}_{k,\,\rho}}\abs{\xoverline{\partial}^*_{g;\,k,\,\rho}\xoverline{\partial} w_{j,\,k,\,\rho}}^2_\omega e^{-\tilde{\varphi}}\,dV_\omega\leq 2\norm{\kappa_{j,\,k,\,\rho}\xoverline{\partial}^*_{g;\,k,\,\rho}\xoverline{\partial}w_{j,\,k,\,\rho}}^2_{\varphi_\rho,\,\omega_{k,\,\rho}} \\
&\leq\frac{4A}{j}\norm{\xoverline{\partial}w_{j,\,k,\,\rho}}_{\varphi_\rho,\,\omega_{k,\,\rho}}\norm{\kappa_{j,\,k,\,\rho}\xoverline{\partial}^*_{g;\,k,\,\rho}\xoverline{\partial}w_{j,\,k,\,\rho}}_{\varphi_\rho,\,\omega_{k,\,\rho}} \\
& \leq \frac{4A\cdot\sqrt{M}}{j}\norm{\kappa_{j,\,k,\,\rho}\xoverline{\partial}^*_{g;\,k,\,\rho}\xoverline{\partial}w_{j,\,k,\,\rho}}_{\varphi_\rho,\,\omega_{k,\,\rho}}{.}
\end{flalign*}
But the second and third inequalities imply that 
\vspace{0.2cm}
\[
\norm{\kappa_{j,\,k,\,\rho}\xoverline{\partial}^*_{g;\,k,\,\rho}\xoverline{\partial}w_{j,\,k,\,\rho}}_{\varphi_\rho,\,\omega_{k,\,\rho}}\leq \frac{2\cdot A\cdot \sqrt{M}}{j}{,}
\]
and since the right hand side is uniformly bounded, the sequence $\{\xoverline{\partial}^*_{g;\,k,\,\rho}\xoverline{\partial}(w_{j,\,k,\,\rho}\rvert_{\tilde{\Omega}_{k,\,\rho}})\}_j$ has a weakly convergent subsequence in 
\vspace{0.2cm}
\[
L^2_{(n,\,1)}(\tilde{\Omega}_{k,\,\rho}, F, \tilde{\varphi}, \omega)^{\oplus p}{.}
\]
We call the weak limit $\zeta_{k,\,\rho}\rvert_{\tilde{\Omega}_{k,\,\rho}}$. But since 
\vspace{0.2cm}
\[
\norm{\zeta_{k,\,\rho}\rvert_{\tilde{\Omega}_{k,\,\rho}}}^2_{\tilde{\varphi},\,\omega}\leq\lim_{j\to \infty}\frac{4\cdot A\cdot\sqrt{M}}{j}\cdot\frac{2\cdot A\cdot\sqrt{M}}{j}=0{,}
\]
the family $\{ \zeta_{k,\,\rho}\rvert_{\tilde{\Omega}_{k,\,\rho}}\}_k$ has a weakly convergent subsequence in 
\vspace{0.2cm}
\[
L^2_{(n,\,1)}(X, F, \tilde{\varphi}, \omega)^{\oplus p}{.}
\] 
Let us call the weak limit $\zeta_\rho$ with the estimate $\norm{\zeta_\rho}^2_{\tilde{\varphi},\,\omega}=0$. \\
\indent Recall that we have checked
\vspace{0.2cm} 
\[
\xoverline{\partial}\xoverline{\partial}^*_{g;\,k,\,\rho}w_{j,\,k,\,\rho}+\xoverline{\partial}^*_{g;\,k,\,\rho}\xoverline{\partial}w_{j,\,k,\,\rho}+C_\rho w_{j,\,k,\,\rho}=v{,}
\]
and by choosing subsequences successively, we may assume without loss of generality that the weak limits $w_{k,\,\rho}$, $\xi_{k,\,\rho}$, and $\zeta_{k,\,\rho}$ are obtained along the same subsequence of $j$, and similarly $w_\rho$, $\xi_\rho$, and $\zeta_\rho$ are obtained along the same subsequence of $k$. Thus, by taking the weak limit in the above equation, the equality $\xoverline{\partial}\xi_\rho+\zeta_\rho+C_\rho w_\rho=v$ holds on $\Omega_\rho$ as distributions. Indeed this equality holds as distributions on the whole $X$ by the following lemma.

\vspace{0.2cm}
\begin{lemma}{\cite[Lemma 6.9]{Dem82}} \label{dislem}
Let $\Omega$ be a domain in $\mathbb{C}^n$, $n\geq1$. Let $Z$ be a proper analytic subset of $\Omega$. Let $u$ be an $L^2$-integrable measurable function and $v$ be a $(0, 1)$-form with $L^1$-integrable measurable coefficients. Assume that $\xoverline{\partial}u =v$ as distributions on $\Omega\setminus Z$. Then this equality holds as distributions on the whole $\Omega$.
\end{lemma}
To apply the above lemma, we claim that our $v= f\wedge\,\xoverline{\partial}(\xoverline{g}/\, \abs{g}^2)$ satisfies $v\in L^1_{\mathrm{loc}}$. Take any small enough relatively compact domain $K$ in $X$. Then it is enough to check that
\[
\int_K \frac{\abs{f}_\omega}{\abs{g}^2}\, dV_\omega< \infty{.}
\]
But this follows from the Cauchy-Schwarz inequality
\[
\int_K \frac{\abs{f}_\omega}{\abs{g}^2}\, dV_\omega\leq \left( \int_K \frac{\abs{f}^2_\omega e^{-\psi +\alpha q\eta}}{\abs{g}^{2\alpha q+ 2}}\, dV_\omega\right)^{1/\, 2} \left(\int_K \abs{g}^{2\alpha q- 2} e^{\psi- \alpha q\eta}\, dV_\omega\right)^{1/\, 2}{.}
\]
The first term in the right hand side is finite by our assumption
\[
\int_X \frac{\abs{f}^2_\omega e^{-\psi- \eta}}{(\abs{g}^2 e^{-\eta})^{\alpha q+ 1}}\, dV_\omega< \infty{.}
\]
The second term in the right hand side is also finite since we may assume $\psi- \alpha q\eta$ is a plurisubharmonic function, which is locally bounded above, by the curvature assumption $\sqrt{-1}\partial\xoverline{\partial}\psi\geq \alpha q \cdot \sqrt{-1}\partial\xoverline{\partial}\eta$. Hence, $v\in L^1_{\mathrm{loc}}$. \\
\indent Furthermore, since 
\vspace{0.2cm}
\[
\norm{\xi_\rho}_{\tilde{\varphi},\,\omega}\leq\sqrt{M}{,}\,\,\,\norm{\zeta_\rho}_{\tilde{\varphi},\,\omega}=0{,}\,\,\,\text{and}\,\,\,\norm{C_\rho w_\rho}_{\tilde{\varphi},\,\omega}\leq\sqrt{2\cdot C_\rho\cdot M}
\]
for $\rho\leq\rho^*$, the sequences $\{\xi_\rho\}_\rho$, $\{\zeta_\rho\}_\rho$, and $\{C_\rho w_\rho\}_\rho$ all possess weakly convergent subsequences in 
\vspace{0.2cm}
\[
L^2_{(n,\,0)}(X, F, \tilde{\varphi}, \omega)^{\oplus p}{,}\,\,\,L^2_{(n,\,1)}(X, F, \tilde{\varphi} ,\omega)^{\oplus p}{,}\,\,\,\text{and}\,\,\,L^2_{(n,\,1)}(X, F, \tilde{\varphi}, \omega)^{\oplus p}{,}
\]
respectively. We denote their respective weak limits as $\xi$, $\zeta$, and $\theta$. By choosing subsequences successively, we may assume without loss of generality that their respective weak limits are obtained along the same subsequence of $\rho$. Then clearly, the equality $\xoverline{\partial}\xi+\zeta+\theta=v$ holds as distributions. Indeed, we can check that $\zeta=\theta=0$ as distributions. This follows from
\vspace{0.2cm}
\[
\lim_{\rho\to0}\norm{C_\rho w_\rho}_{\tilde{\varphi},\,\omega}\leq\lim_{\rho\to0}\sqrt{2\cdot C_\rho\cdot M}=0\,\,\,\text{and}\,\,\,\norm{\zeta_\rho}_{\tilde{\varphi},\,\omega}=0{.}
\]
\end{step}
\vspace{0.5cm}

\begin{step}[End of the proof of the main theorem] \label{Step6}
Since $\zeta=\theta=0$, we can check $\xoverline{\partial}\xi=v$. Now define
\vspace{0.2cm}
\[
h=f\cdot(\xoverline{g}/\,\abs{g}^2)-\xi
\]
on the domain $X\setminus\{g_1=0\}$. We can readily check $g\cdot h=f$ since $g\cdot \xi=0$, and $\xoverline{\partial}h=v-v=0$. Hence our main theorem holds on the domain $\abs{g}\neq0$, by taking the holomorphic representative of $h$, with the estimate
\vspace{0.2cm}
\[
\int_X\frac{\abs{h}_\omega^2e^{-\psi}}{(\abs{g}^2e^{-\eta})^{\alpha q}}\,dV_\omega\leq \Bigl(2+\frac{4(\alpha+1)}{(\alpha-1)^2}\Bigr) \int_X\frac{\abs{f}_\omega^2e^{-\psi-\eta}}{(\abs{g}^2e^{-\eta})^{\alpha q+1}}\,dV_\omega{.}
\]
This estimate follows from the fact that $\abs{h}_\omega^2\leq 2(\abs{\xi}_\omega^2+\abs{f\cdot (\xoverline{g}/\,\abs{g}^2)}_\omega^2)$ and the following lemma:
\vspace{0.2cm}
\begin{lemma}
We have the following estimate for $\xi$:
\[
\int_X\frac{\abs{\xi}_\omega^2e^{-\psi}}{(\abs{g}^2e^{-\eta})^{\alpha q}}\,dV_\omega\leq M= \frac{2(\alpha+1)}{(\alpha-1)^2}\int_X\frac{\abs{f}_\omega^2e^{-\psi-\eta}}{(\abs{g}^2e^{-\eta})^{\alpha q+1}}\,dV_\omega{.}
\]
\end{lemma}
\begin{proof}
Fix a $\rho_0> 0$ which is smaller than $\rho^*$, and assume $\rho< \rho_0$ here so that $e^{-\varphi_{\rho_0}}\leq e^{-\varphi_{\rho}}$. \\
\indent Recall that we defined $\xi_{k,\,\rho}$ as a weak limit in $L^2_{(n,\,0)}(\tilde{\Omega}_{k,\,\rho}, F, \tilde{\varphi}, \omega)^{\oplus p}$ given by a subsequence of $\{\xoverline{\partial}^*_{g;\,k,\,\rho}w_{j,\,k,\,\rho}\rvert_{\tilde{\Omega}_{k,\,\rho}}\}_j$. But since both $\tilde{\varphi}$ and $\varphi_{\rho_0}$ are smooth on the closure of $\tilde{\Omega}_{k,\, \rho}$ and $\tilde{\Omega}_{k,\,\rho}$ is relatively compact in $\Omega_\rho$, we can also say $\xi_{k,\,\rho}$ is a weak limit in $L^2_{(n,\,0)}(\tilde{\Omega}_{k,\,\rho}, F, \varphi_{\rho_0}, \omega)^{\oplus p}$ given by the same subsequence of $\{\xoverline{\partial}^*_{g;\,k,\,\rho}w_{j,\,k,\,\rho}\rvert_{\tilde{\Omega}_{k,\,\rho}}\}_j$. This implies that
\[
\norm{\xi_{k,\,\rho}}^2_{\varphi_{\rho_0},\,\omega}= \int_{\tilde{\Omega}_{k,\,\rho}} \abs{\xi_{k,\,\rho}}^2_{\omega}e^{-\varphi_{\rho_0}}\,dV_\omega\leq M{.}
\]
Take any relatively compact domain $K$ in $\Omega_\rho= X\setminus Z_\rho$. For a similar reason, we can say that $\xi_\rho\rvert_K$ is a weak limit in $L^2_{(n,\,0)}(K, F, \varphi_{\rho_0}, \omega)^{\oplus p}$ given by a subsequence of $\{\xi_{k,\,\rho}\rvert_K\}_k$. Since $K\subset \tilde{\Omega}_{k,\,\rho}$ for every sufficiently large $k$,
\[
\norm{\xi_\rho\rvert_K}^2_{\varphi_{\rho_0},\,\omega}= \int_K \abs{\xi_\rho}^2_{\omega}e^{-\varphi_{\rho_0}}\, dV_\omega\leq M\,\,\,\text{and}\,\,\,\int_X \abs{\xi_\rho}^2_{\omega}e^{-\varphi_{\rho_0}}\, dV_\omega\leq M{.}
\]
The second inequality holds because $K\Subset\Omega_\rho$ is arbitrary. Now take any relatively compact domain $K'$ in $\Omega_{\rho_0}$. Again, we can say that $\xi\rvert_{K'}$ is a weak limit in $L^2_{(n,\,0)}(K', F, \varphi_{\rho_0}, \omega)^{\oplus p}$ given by a subsequence of $\{\xi_\rho\rvert_{K'}\}_\rho$. This implies
\[
\int_{K'} \abs{\xi}^2_{\omega}e^{-\varphi_{\rho_0}}\, dV_\omega\leq M\,\,\,\text{and}\,\,\,\int_X \abs{\xi}^2_{\omega}e^{-\varphi_{\rho_0}}\, dV_\omega\leq M{.}
\]
The second inequality holds because $K'\Subset \Omega_{\rho_0}$ is arbitrary. By the Lebesgue monotone convergence theorem, we can check that
\[
\int_X \frac{\abs{\xi}^2_{\omega}e^{-\psi}}{(\abs{g}^2e^{-\eta})^{\alpha q}}\, dV_\omega\leq M{.}
\]
\end{proof}
To extend the equality $g\cdot h=f$ to the whole $X$, we just apply the Riemann extension theorem to $h$.
\end{step}

\vspace{0.5cm}

\section{Proof of \Cref{MAIN THEOREM 2}}

\Cref{MAIN THEOREM 2} is a direct consequence of \Cref{MAIN THEOREM 1} obtained by setting appropriate line bundles and parameters. \\
\indent Assume first that $q_0= \min\{n, p-1\} > 0$. Let $q\geq q_0$ be an integer. We apply \Cref{MAIN THEOREM 1} by substituting $F$ with $(q+1)E+F$ and $E$ with $E$. We equip $E$ with the given singular hermitian metric $e^{-\eta}$, and $(q+1)E+F$ with the singular hermitian metric $e^{-\psi'}$, where $\psi'= \psi+ (q+1)\eta$. \\
\indent We choose the rational constant $\alpha= (q+1)/\,q_0> 1$. The curvature condition of \Cref{MAIN THEOREM 1} is verified as follows:
\vspace{0.2cm}
\[
\sqrt{-1}\partial\xoverline{\partial}\psi'= \sqrt{-1}\partial\xoverline{\partial}\psi+ (q+1)\sqrt{-1}\partial\xoverline{\partial}\eta\geq (q+1)\sqrt{-1}\partial\xoverline{\partial}\eta= \alpha q_0 \sqrt{-1}\partial\xoverline{\partial}\eta{,}
\]
where we used the assumptions $\sqrt{-1}\partial\xoverline{\partial}\psi \geq 0$ and $\sqrt{-1}\partial\xoverline{\partial}\eta \geq 0$. Since
\vspace{0.2cm}
\[
\int_X\frac{\abs{f}_\omega^2e^{-\psi'-\eta}}{(\abs{g}^2e^{-\eta})^{\alpha q_0+1}}\,dV_\omega = \int_X\frac{\abs{f}_\omega^2e^{-\psi-(q+2)\eta}}{(\abs{g}^2e^{-\eta})^{q+2}}\,dV_\omega <\infty{,}
\]
by \Cref{MAIN THEOREM 1}, there exist $p$ sections $h_1, \cdots, h_p \in H^0(X, K_X + (q+1)E + F)$ such that $\sum_k h_k g_k = f$ and 
\vspace{0.2cm}
\[
\int_X\frac{\abs{h}_\omega^2e^{-\psi'}}{(\abs{g}^2e^{-\eta})^{\alpha q_0}}\,dV_\omega \leq \Bigl(2+\frac{4(\alpha+1)}{(\alpha-1)^2}\Bigr) \int_X\frac{\abs{f}_\omega^2e^{-\psi'-\eta}}{(\abs{g}^2e^{-\eta})^{\alpha q_0+1}}\,dV_\omega{.}
\]
Since $q\geq q_0$, we can check that
\[
2+ \frac{4(\alpha+ 1)}{(\alpha- 1)^2}= 2+ \frac{4(q+ q_0+ 1)q_0}{(q- q_0+ 1)^2}\leq 8q^2+ 4q+ 2{.}
\]
This precisely gives the desired estimate:
\vspace{0.2cm}
\[
\int_X\frac{\abs{h}_\omega^2e^{-\psi-(q+1)\eta}}{(\abs{g}^2e^{-\eta})^{q+1}}\,dV_\omega \leq (8q^2+4q+2)\int_X \frac{\abs{f}_\omega^2e^{-\psi -(q+2)\eta}}{(\abs{g}^2e^{-\eta})^{q+2}}\,dV_\omega{.}
\]
\indent If $q_0= 0$, then $p= 1$. Apply the $q= 0$ case of \Cref{MAIN THEOREM 1} by substituting $F$ with $(q+ 1)E+ F$ endowed with the metric whose weight is $\psi+ (q+ 1)\eta$. Then $h_1= f/\, g_1$ extends holomorphically and
\[
\int_X \abs{h_1}^2_\omega e^{-\psi- (q+ 1)\eta}\, dV_\omega= \int_X \frac{\abs{f}^2_\omega e^{-\psi- (q+ 2)\eta}}{\abs{g_1}^2 e^{-\eta}}\, dV_\omega{.}
\]

\vspace{0.5cm}

\section{Proof of \Cref{Corollary 1} and \Cref{Corollary 2}}
\subsection{PROOF OF \Cref{Corollary 1}}
The proof of \Cref{Corollary 1} follows the one in the version in \cite{Var08} with minor modifications. \\
\indent Assume that $G= (G_1, \cdots, G_p)$ is not identically zero. After relabeling, we may assume $G_1\not\equiv 0$. By taking $G_{p+1}=\cdots=G_{n+1}=0$, we may assume $q=n$. Take $\alpha=(n+k)/n$ so that $\alpha q=n+k$. We want to apply our main theorem with $F=(n+k)L+H$, $E=L$, and $g_i=G_i$. Fix a metric $e^{-\eta}$ for $L$ having non-negative curvature current, for example one can take $\eta=\log\,\abs{G_1}^2$, and let $\psi=\varphi+ (n+k)\eta$. Then
\vspace{0.2cm}
\[
\sqrt{-1}\partial\xoverline{\partial}\psi-\alpha q\cdot \sqrt{-1}\partial\xoverline{\partial}\eta=\sqrt{-1}\partial\xoverline{\partial}\varphi\geq0{.}
\]
Suppose that $f\in H^0(X,\mathcal{O}_X((n+k+1)L+H+K_X)\otimes\mathcal{J}_{k+1})$. By our main theorem, there are $h_1,\cdots,h_p\in H^0(X,\mathcal{O}_X((n+k)L+H+K_X))$ such that $\sum_l h_lG_l=f$. Moreover,
\vspace{0.2cm}
\begin{flalign*}
&\int_X\frac{\abs{h}_\omega^2e^{-\varphi}}{\abs{g}^{2(n+k)}}\,dV_\omega=\int_X\frac{\abs{h}_\omega^2e^{-\psi}}{(\abs{g}^2e^{-\eta})^{n\alpha}}\,dV_\omega \leq \Bigl(2+\frac{4(\alpha+1)}{(\alpha-1)^2}\Bigr)\int_X\frac{\abs{f}_\omega^2 e^{-\psi-\eta}}{(\abs{g}^2e^{-\eta})^{n\alpha+1}}\,dV_\omega \\
&=\frac{2(k^2 + 2nk + 4n^2)}{k^2}\int_X\frac{\abs{f}_\omega^2e^{-\varphi}}{\abs{g}^{2(n+k+1)}}\,dV_\omega<\infty{.}
\end{flalign*}
Thus $h_l\in\mathcal{J}_k$ locally.


\vspace{0.5cm}

\subsection{PROOF OF \Cref{Corollary 2}}
The proof is almost identical to the version in \cite{Siu05}, but we include the proof for the reader's convenience. \\
\indent For $0\leq l<a$, let $H_l= (b+l)F -K_X$ and $L= aF$. Let $G_1,\cdots,G_p$ be a basis of $H^0(X, L)$. Let $I_1,\cdots,I_q$ be a basis of $H^0(X, bF- K_X)$. We give $H_l$ the metric
\vspace{0.2cm}
\[
e^{-\varphi_l}= \frac{1}{(\sum^p_{j=1}\abs{G_j}^{2l/\,a})(\sum^q_{j=1}\abs{I_j}^2)}{.}
\]
Since both multiplier ideal sheaves
\vspace{0.2cm}
\[
\mathcal{J}_{k+1}=\mathcal{J}(e^{-\varphi_l}\abs{G}^{-2(n+k+1)})\,\,\,\text{and}\,\,\,\mathcal{J}_{k}=\mathcal{J}(e^{-\varphi_l}\abs{G}^{-2(n+k)})
\]
are just $\mathcal{O}_X$ due to the global generation of $aF$ and $bF-K_X$, it follows from \Cref{Corollary 1} that
\vspace{0.2cm}
\[
H^0(X, (n+k+1)L+ H_l+K_X)= \sum^p_{j=1} G_j\cdot H^0(X, (n+k)L+H_l+K_X)
\]
for $k\geq 1$ and $0\leq l< a$, which means that 
\vspace{0.2cm}
\[
H^0(X, ((n+k+1)a+l+b)F)= \sum^p_{j=1} G_j\cdot H^0(X, ((n+k)a+l+b)F)
\]
for $k\geq 1$ and $0\leq l< a$. Thus $\bigoplus^{(n+2)a+b-1}_{m=0} H^0(X, mF)$ generates the ring 
\vspace{0.2cm}
\[
\bigoplus^\infty_{m=0} H^0(X, mF){.}
\]

\bibliographystyle{amsalpha}
\bibliography{ref}

\end{document}